\documentclass[leqno,12pt]{article}
\usepackage{amsfonts,amsmath,amssymb}
\usepackage{amsthm}
\usepackage{enumerate}
\usepackage{graphics}
\usepackage{float}%
\usepackage{enumerate}%
\usepackage{hyperref}%

\newcommand{\Lie}	{\operatorname{Lie}}
\newcommand{\Hom}	{\operatorname{Hom}}
\newcommand{\Ker}	{\operatorname{Ker}}
\newcommand{\Ind}	{\operatorname{Ind}}
\newcommand{\Ad}	{\operatorname{Ad}}
\newcommand{\ang}[2]{\langle #1,#2\rangle}

\newcommand{\scirc}	{\raise2pt\hbox{${}_\circ$}}

\newcommand{\id}	{\operatorname{id}}
\newcommand{\Lrep}   {{\widehat{L}_{\operatorname{f}}}} 
\newcommand{\rhoc}   {{\rho_{\mathfrak{t}}}}
\newcommand{\rhon}   {{\rho_{\mathfrak{n}}}}
\newcommand{\rhonbar}{{\rho_{\bar{\mathfrak{n}}}}}

\newcommand{\ipgz}  {I_P^G(\zeta)}
\newcommand{\vzp}  {V_{\zeta,P}}
\newcommand{\vzdp}  {V_{\zeta^{\ast},P}}
\newcommand{\vzzp}  {\vzp \boxtimes V_{\zeta^*, \overline{P}}}

\newtheorem{thmalph}{Theorem}

\newtheorem{coralph}[thmalph]{Corollary}

\newtheorem{theorem}{Theorem}[section]
\newtheorem{proposition}[theorem]{Proposition}
\newtheorem{corollary}[theorem]{Corollary}
\newtheorem{lemma}[theorem]{Lemma}

\theoremstyle{remark}
\newtheorem{remark}[theorem]{Remark}
\theoremstyle{definition}
\newtheorem{example}[theorem]{Example}
\newtheorem{definition}[theorem]{Definition}
\providecommand{\bysame}{\makebox[3em]{\hrulefill}\thinspace}
\numberwithin{equation}{section}
\begin{document}
\title{Finite multiplicity theorems
\\ for induction and restriction
}
\author{%
 Toshiyuki Kobayashi
 \hspace{0.3mm} and Toshio Oshima
}
\date{} 

\maketitle

\begin{abstract}
We find upper and lower bounds of the multiplicities of
irreducible admissible representations $\pi$
 of  semisimple Lie groups $G$ occurring in the
induced representations $\operatorname{Ind}_H^G\tau$ from 
irreducible representations $\tau$
 of closed subgroups $H$.  
As corollaries, we establish 
geometric criteria for finiteness of 
the dimension of $\Hom_G(\pi,\Ind_H^G \tau)$
(induction)
 and of $\Hom_H(\pi|_H,\tau)$ (restriction)
 by means of the real flag variety $G/P$,
 and
discover that uniform boundedness property of these multiplicities
is independent of real forms and characterized
 by means of the complex flag variety.
\end{abstract}



\section{Introduction}
\label{sec:1}

The motivation of this work is the following fundamental
questions in non-commutative harmonic analysis beyond symmetric spaces
and branching problems of infinite-dimensional
representations of real reductive Lie groups:

1. (Induction)  What is the `most general setting'
 of homogeneous spaces $G/H$ in which we could expect reasonable
 and detailed analysis of function spaces on $G/H$?

2. (Restriction) What is the  `most general setting'
of pairs $(G, H)$ for which we could expect
reasonable and detailed analysis of branching laws
 of the restriction of (arbitrary) irreducible representations of $G$ to $H$?

We shall give an answer to these questions from the viewpoint of
 multiplicities of irreducible representations.

\vskip 1pc

Let $G$ be a connected real semisimple Lie group with finite center, 
and $H$ a closed (not necessarily, reductive) subgroup 
 with at most finitely many connected components.
(It is easy to see
 that the results of this article remain true
 if we replace connected semisimple Lie groups $G$
 by linear reductive groups.)

We consider the following two geometric conditions:
\begin{align}
&\text{There exists an open $H$-orbit on the real flag variety $G/P$.}
\label{eqn:HP}
\tag{HP}
\\
&\text{There exists an open $H_c$-orbit on the complex flag variety $G_c/B$.  
}
\label{eqn:HB}
\tag{HB}
\end{align}
Here $P$ is a minimal parabolic subgroup of $G$,
$B$ is a Borel subgroup of a complex Lie group $G_c$
 with the complexified Lie algebra 
${\mathfrak g}_c= {\mathfrak{g}} \otimes_{\mathbb{R}}{\mathbb{C}}$, 
 and $H_c$ a complex subgroup with Lie algebra 
 ${\mathfrak{h}}_c={\mathfrak{h}} \otimes_{\mathbb{R}}{\mathbb{C}}$,
 where ${\mathfrak {g}}$ and ${\mathfrak {h}}$
 are the Lie algebras of $G$ and $H$, 
respectively. 
The condition \eqref{eqn:HB} is equivalent to
 that $G_c/H_c$ is {\it{spherical}}
 (i.e. $B$ has an open orbit on $G_c/H_c$)
 when $G \supset H$ are defined algebraically.  
Similarly, we call $G/H$ is {\it {real spherical}} \cite{Kb5}
 if \eqref{eqn:HP} is satisfied (i.e. $P$ has an open orbit on $G/H$),
 see Remark \ref{rem:2.4} (4)
 for equivalent definitions.

An analogous notation $P_H \subset H$
 and $B_H \subset H_c$
will be applied when $H$ is reductive.  
In this case we can consider also the following two conditions:
\begin{align}
&\text{There exists an open $P_H$-orbit on the real flag variety $G/P$.}
\label{eqn:PP}
\tag{PP}
\\
&\text{There exists an open $B_H$-orbit on the complex flag variety $G_c/B$.  
}
\label{eqn:BB}
\tag{BB}
\end{align}
Clearly,
 these four conditions 
 on the pair $(G,H)$ do not depend on the choice of parabolics, 
 coverings or connectedness of the groups,
 but are determined locally,
 namely, 
 only
 by the Lie algebras ${\mathfrak {g}}$
 and ${\mathfrak {h}}$.   
An easy argument (see Lemma \ref{lem:HPB}) 
 shows that the following implications hold.  
Here we consider \eqref{eqn:PP} and \eqref{eqn:BB}
 when $H$ is reductive:
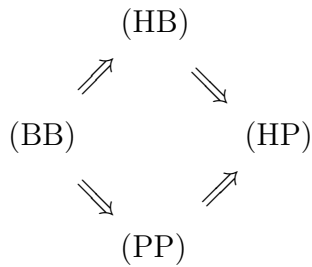
\begin{alignat*}{5}
& && &&\eqref{eqn:HB} && &&
\\
& &&\rotatebox{45}{$\Longrightarrow$}  && && \rotatebox{135}{$\Longleftarrow$}  &&
\\
&\eqref{eqn:BB} && && && && \eqref{eqn:HP}
\\
& && \rotatebox{135}{$\Longleftarrow$} && && \,\,\rotatebox{45}{$\Longrightarrow$} &&
\\
& && &&\eqref{eqn:PP}&& &&
\end{alignat*}
None of the converse implications is true:  
\begin{example}
[{\cite[Example 2.8.6]{Kb5}}]
Let $(G,H)$ be a triple product pair
$(
{}^{\backprime\!} G\times {}^{\backprime\!} G \times {}^{\backprime\!} G,
\Delta {}^{\backprime\!} G)$
 with ${}^{\backprime\!}G$ being a simple Lie group.  
In \cite{Kb5}
 we gave the following classification:
 \eqref{eqn:HP} holds
 if and only if ${}^{\backprime\!}G$ is compact
 or 
$
     {}^{\backprime\!}{\mathfrak {g}} \simeq {\mathfrak {so}}(n,1),
$
 \eqref{eqn:HB} holds
 if and only if 
$
      {}^{\backprime\!}{\mathfrak {g}} \simeq {\mathfrak {su}}(2), 
$ 
$
     {\mathfrak {sl}}(2,{\mathbb{R}}), 
$
 or 
$
     {\mathfrak {sl}}(2,{\mathbb{C}}); 
$
 \eqref{eqn:PP} holds
 if and only if ${}^{\backprime\!}G$ is compact;
 \eqref{eqn:BB} never holds.  
The condition \eqref{eqn:HP}
 in the triple product pair has laid a solid foundation
 of concrete analysis of the tensor product 
 of two representations
 (see \cite{BeR}
 for ${}^{\backprime\!}G=SL(2,{\mathbb{R}})$;
\cite{CKOP} for ${}^{\backprime\!} G=SO(n,1)$, 
 for instance).    
\end{example}
It should be noted 
 that the two conditions \eqref{eqn:HB} and \eqref{eqn:BB} depend 
only on the complexifications $({\mathfrak {g}}_c, {\mathfrak {h}}_c)$.  
It is known by the work of Brion, Kr\"amer, and
 Vinberg--Kimelfeld
 \cite{Br, Kr2, Kr3,VK}
 that the geometric condition 
 \eqref{eqn:HB} characterizes the multiplicity-free
 property
 of irreducible (algebraic)
 {\it{finite dimensional}} representations 
 $\pi$
 in the induced representation $\Ind_H^G \tau$ 
 with $\dim \tau =1$
 (i.e. $(G,H)$ is a {\it{Gelfand pair}}),  
 and that the condition \eqref{eqn:BB} 
 characterizes the multiplicity-free property of the restriction $\pi|_H$ 
 with respect to $G \downarrow H$
 (i.e. $(G,H)$ is a {\it{strong Gelfand pair}}).  
An extensive research
 has been made in the decades
 in connection with algebraic group actions,
 invariant theory,
 and symplectic geometry
 among others
 (e.g. \cite{Vi}), 
 but mostly in the framework of algebraic 
 (finite dimensional) representations
 or in the case of compact subgroups $H$.

These beautiful classic results 
 may play a guiding principle
 in considering what a natural generalization
 would be for non-compact subgroups $H$
 (or for non-Riemannian homogeneous spaces $G/H$), 
 however,
 only a complete change of machinery 
 has enabled us 
 to prove finite/bounded
 multiplicity results
 for admissible representations.  
Namely,
 in order to overcome analytic difficulties
 arising from non-compact subgroups $H$
 and from infinite dimensional representations, 
 we employ the theory of a system
 of partial differential equations
 with regular singularities,
 for which micro-local analysis
 gives a canonical method.  
Thus we establish in this paper 
 that the above four geometric conditions 
 \eqref{eqn:HP}, \eqref{eqn:HB}, \eqref{eqn:PP}, 
 and \eqref{eqn:BB} characterize 
 finiteness$/$boundedness of the multiplicities
 of the induction$/$restriction
 for admissible representations
 of real reductive groups, 
 respectively
 (see Theorems \ref{thm:A}--\ref{thm:D} below).  

For a precise statement of our results,
 let $\widehat G_{\operatorname{ad}}$ denote 
 the set of equivalence classes
 of irreducible admissible 
 representations of $G$
 (see Definition \ref{def:adm}), 
 and $\widehat{G}_{\operatorname{f}}$
 that of irreducible finite dimensional representations
 of $G$.
We write $c_{\mathfrak{g},K}(\pi,\Ind_H^G\tau)$ 
 for the multiplicity
 of the underlying $(\mathfrak{g},K)$-module $\pi_K$ of $\pi \in \widehat G_{\operatorname{ad}}$
 occurring in the space of 
 sections of the $G$-homogeneous vector bundle over $G/H$
 associated to $\tau\in \widehat H_{\operatorname{f}}$
 (the topology of $\Ind_H^G \tau$ is not the main issue here
 owing to analytic elliptic regularity).  

\begin{thmalph} 
[finite multiplicity theorem for induction] 
\label{thm:A}
~
\par
\noindent
{\rm 1)}
If \eqref{eqn:HP} holds, 
then
$
  c_{\mathfrak{g},K}(\pi,\Ind_H^G\tau)<\infty
$
 for any $\pi \in \widehat{G}_{\operatorname{ad}}$ and any $\tau \in \widehat H_{\operatorname{f}}$.  

\noindent
{\rm 2)} 
Suppose that $G$, $H$
 and $\tau$ are defined algebraically over ${\mathbb{R}}$. 
If \eqref{eqn:HP} fails,  
then for  any algebraic representation $\tau$ of $H$
 there exists $\pi\in\widehat G_{\operatorname{ad}}$ such that
$
  c_{\mathfrak{g},K}(\pi,\Ind_H^G\tau)=\infty
$.
\end{thmalph}
An upper bound formula
 of the multiplicities is presented in\/ {\rm Theorem~\ref{thm:2.1}}, 
which is strong enough
 to give a proof of uniformly bounded
 multiplicity results
 under stronger assumptions
  (Theorems \ref{thm:B} and \ref{thm:D} below), 
 and thus plays a central role throughout the paper.
The algebraic assumption
 in the second statement of Theorem \ref{thm:A}
 is crucial.  
A counterexample without the algebraic assumption
 is illustrated in Example \ref{ex:alg}.

Concerning the uniform boundedness 
 of the multiplicities for the induced representation,
 we may consider the following three kinds
 of conditions:
\begin{align}
 \underset{\tau \in \widehat{H}_{\operatorname{f}}}\sup \,
 \underset{\pi \in \widehat{G}_{\operatorname{ad}}}\sup \,
 \frac{1}{\dim\tau} {c_{\mathfrak{g},K}(\pi,\Ind_H^G\tau)}
 &< \infty.  
\label{eqn:HG1}
\\
 \sup\limits_{\substack{\tau\in\widehat{H}_{\operatorname{f}}\\ \dim\tau=1}}
 \underset{\pi \in \widehat{G}_{\operatorname{ad}}}\sup \,
 c_{\mathfrak{g},K}(\pi,\Ind_H^G\tau)&<\infty.  
\label{eqn:HG2}
\\
  \underset{\pi \in \widehat{G}_{\operatorname{ad}}}\sup
 c_{\mathfrak{g},K}(\pi,C^{\infty}(G/H)) &< \infty.  
\label{eqn:HG3}
\end{align}

Clearly,
\eqref{eqn:HG1} $\Rightarrow$ 
\eqref{eqn:HG2} $\Rightarrow$
\eqref{eqn:HG3}.

Needless to say, 
$\widehat{G}_{\operatorname{ad}}$ and $\widehat{H}_{\operatorname{ad}}$ depend
heavily on real forms $(G,H)$ of $(G_c,H_c)$.  
Surprisingly,
 we discovered 
 in the following theorem 
 (and also Theorem \ref{thm:D})
 that 
 the uniform boundedness condition of the multiplicities is
determined only by the complexified Lie algebras
$(\mathfrak{g}_c,\mathfrak{h}_c)$.  

\begin{thmalph}
[uniformly bounded theorem of multiplicities for induction]
\label{thm:B}
~
\par\noindent
{\rm{1)}}\enspace
The condition \eqref{eqn:HB} implies
 \eqref{eqn:HG1} 
 $($hence, \eqref{eqn:HG2} and  \eqref{eqn:HG3}, too$)$.  
\par\noindent
{\rm{2)}}\enspace
Suppose $(G,H)$ is defined algebraically over ${\mathbb{R}}$.  
Then the conditions \eqref{eqn:HB}, 
\eqref{eqn:HG1}, 
and \eqref{eqn:HG2},  
are all equivalent.  
Further, if $H$ is reductive, 
 then \eqref{eqn:HG3} 
 is equivalent to these conditions, too.  
\end{thmalph}
\begin{remark}
\label{rem:A}
Theorem \ref{thm:B} is classically known 
 for compact Lie group $G$
 even in a stronger form
 \cite{Br, Kr3},
 i.e.
 the upper bound \eqref{eqn:HG3}
 is one. 
In contrast to the compact case, 
 the upper bound \eqref{eqn:HG3}
 is often greater than one
 if $H$ is noncompact.  
For instance,
 if $(G,H)$ is a semisimple symmetric pair
 $(SL(p+q,{\mathbb{R}}), SO_0(p,q))$,
 then the upper bound \eqref{eqn:HG3}
 is no less than $(p+q)!\,/\,p!\, q!$ 
 in view of the contribution
 of the most continuous principal series
 representations for $G/H$
 (see \cite{OS}).  
\end{remark}
\begin{remark}
\label{rem:HBequiv}
It is known that the condition \eqref{eqn:HB}
 is equivalent to the commutativity
 of the ring of $G$-invariant differential operators on $G/H$.  
Further, if $H$ is compact
 then the condition \eqref{eqn:HB} is equivalent
 to that the Riemannian manifold $G/H$ is a weakly symmetric space
in the sense of Selberg.  
\end{remark}
\begin{example}
\label{ex:AB}
{\rm{1)}}\enspace
If $(G,H)$ is a symmetric pair,
 then the condition \eqref{eqn:HB}
 (and therefore \eqref{eqn:HP}) is always  
 fulfilled.  
In particular,
 the uniform bounded estimate \eqref{eqn:HG1}
 holds 
 by Theorem \ref{thm:B}.  
This improves an earlier work of van den Ban \cite {Ba}: 
\begin{equation}
\label{eqn:ban}
    c_{{\mathfrak {g}},K}(\pi,\Ind_H^G \tau)
    <\infty
    \quad
    \text{for any $\pi \in \widehat G_{\operatorname{ad}}$
 and $\tau \in\widehat H_{\operatorname{f}}$, }
\end{equation}
which does not imply the uniform estimate
 with respect to $\pi$.  
In our context, 
 the weaker estimate \eqref{eqn:ban} is derived from a more general geometric
 condition \eqref{eqn:HP}
 by Theorem \ref{thm:A}.  
\par\noindent
{\rm{2)}}\enspace
If $G_c/H_c$ is  spherical
 then any real form $(G,H)$ satisfies \eqref{eqn:HB}.  
There are some few non-symmetric 
 spherical homogeneous spaces $G_c/H_c$ 
  such as 
$
     ({\mathfrak {g}}_c,{\mathfrak {h}}_c)
= ({\mathfrak {sl}}(2n+1,{\mathbb{C}}),
   {\mathfrak {sp}}(n,{\mathbb{C}})), 
$
$
     ({\mathfrak {so}}(2n+1,{\mathbb{C}}),
      {\mathfrak {gl}}(n,{\mathbb{C}})),
$
 and $({\mathfrak {so}}(7,{\mathbb{C}}),
       {\mathfrak {g}}_2({\mathbb{C}}))$, 
 and they have been classified 
 in \cite{Br, Kr3}.  
Further, 
 it was proved in \cite{Kb2I}
 that some non-symmetric, real spherical homogeneous spaces
 $G/H$
 such as 
$
  SU(n,n+1)/Sp(n,{\mathbb{R}}),
$
$
  SU(2p+1,2q)/Sp(p,q),
$
$
  G_2({\mathbb{R}})/SL(3,{\mathbb{R}}), 
$
$
     G_2({\mathbb{R}})/SU(2,1), 
$
 etc. 
 admit discrete series representations
 (i.e. irreducible unitary representations
 that occur in closed subspaces
 of the $L^2$-spaces)
 and that some others 
 like 
$
  SL(2n+1,{\mathbb{R}})/Sp(n,{\mathbb{R}})
$
 do not.  

\par\noindent
{\rm{3)}}\enspace 
If we take $H$ to be a maximal unipotent subgroup $N$,
 then \eqref{eqn:HP} holds by the open Bruhat cell. 
The condition \eqref{eqn:HB} is satisfied
 if and only if $G$ is quasi-split.  
Our general formula (Theorem \ref{thm:2.1}) 
 applied to this special case 
 gives an exact estimate of multiplicities
 of generalized Whittaker vectors
 for generic parameter in comparison with 
 the Kostant--Lynch theory
 (\cite {Ks78, Ly}; see Remark \ref{rem:2.4}).  
\end{example}

In Theorems \ref{thm:A} and \ref{thm:B}
 we have allowed $H$ to be non-reductive
 and $\pi$ to be infinite dimensional,
 but have confined $\tau$ to be finite dimensional.  
In Theorems \ref{thm:C} and \ref{thm:D} below,
 we treat the case 
 where both $\pi$ and $\tau$ are 
 allowed to be infinite dimensional.  
Let denote by $\Hom_H(\ ,\ )$ the space of continuous $H$-intertwining
operators.

\begin{thmalph}
[finite multiplicity theorem for restriction]
\label{thm:C}
Assume $H$ is reductive in $G$.  

\noindent
{\rm 1)} If \eqref{eqn:PP} holds,
 then
$
  \dim\Hom_H(\pi|_H,\tau) < \infty
$
 for any $\pi \in \widehat{G}_{\operatorname{ad}}$
 and
 for any $\tau \in \widehat{H}_{\operatorname{ad}}$.

\noindent
{\rm 2)}
Suppose $(G,H)$ is defined algebraically over ${\mathbb{R}}$.  
If \eqref{eqn:PP} fails,
then there exist $\pi \in \widehat {G}_{\operatorname{ad}}$
 and $\tau \in \widehat {H}_{\operatorname{ad}}$
 such that
$
 \dim\Hom_H(\pi|_H,\tau) = \infty
$.
\end{thmalph}
\begin{remark}
\label{rem:HP}
If $H=K$ then its minimal parabolic subgroup $P_H$
 coincides with $K$ itself and the assumption \eqref{eqn:PP} 
 is automatically satisfied because $G=KP$.  
In this simplest case, 
 any irreducible representation $\tau$ is \textit{finite dimensional}
 and our argument for Theorem \ref{thm:C} 1)
 using hyperfunctions recovers 
 so-called Casselman's subrepresentation theorem 
 for which an algebraic proof using Jacquet functors \cite{C}
 is also known.  
Our proof of  
Theorem \ref{thm:C} 1) is given  in Section \ref{sec:2},
which in this special case 
 includes an analytic proof to an earlier work of Harish-Chandra
 \cite{HC}
 that every irreducible quasi-simple representation of $G$
 has finite $K$-multiplicities 
 (cf. Section \ref{sec:2}), 
 for which an algebraic proof is also known
 (cf. \cite[Chapter 3]{W}).  
\end{remark}

Concerning uniform boundedness for the multiplicities of the restriction,
 we consider the following three kinds of conditions:
\begin{alignat}{2}
 &\underset{\tau \in \widehat{H}_{\operatorname{ad}}}\sup \;
 \underset{\pi \in \widehat{G}_{\operatorname{ad}}}\sup  \,
 &&\dim\Hom_H(\pi|_H,\tau) < \infty.   
\label{eqn:GH1}
\\
 &\sup_{\tau \in \widehat {H}_{\infty}} 
   \sup_{\pi \in \widehat {G}_{\infty}}
    &&\dim \Hom_H(\pi|_H, \tau)<\infty.
\label{eqn:GH2}
\\
&\sup_{\tau \in \widehat H_{\operatorname{f}}}\sup_{\pi \in \widehat G_{\operatorname{f}}}
&&\dim \Hom_H(\pi|_H, \tau)<\infty.  
\label{eqn:GH3}
\end{alignat} 
Here $\widehat G_{\infty}$ ($\subset \widehat G_{\operatorname{ad}}$)
 denotes the set of equivalence classes
 of irreducible smooth admissible representations.  
 Clearly,
 \eqref{eqn:GH1} $\Rightarrow$ \eqref{eqn:GH2} $\Rightarrow$ \eqref{eqn:GH3}.  
\begin{thmalph}
[uniformly bounded theorem of multiplicities for restriction]
\label{thm:D}
Assume $H$ is reductive.
\par\noindent
{\rm{1)}}\enspace
The condition \eqref{eqn:BB} implies \eqref{eqn:GH1}
 $($hence, \eqref{eqn:GH2}
 and \eqref{eqn:GH3}, too$)$.  

\par\noindent
{\rm{2)}} \enspace
Assume $(G,H)$ is defined algebraically over ${\mathbb{R}}$.  
Then \eqref{eqn:BB}, 
\eqref{eqn:GH1}, 
 \eqref{eqn:GH2}, 
 and \eqref{eqn:GH3} are all equivalent. 
\end{thmalph}

\begin{example}
\label{ex:D}
{\rm{1)}}\enspace
Owing to the classification \cite{Kr2}, 
 the condition \eqref{eqn:BB} is equivalent to
 that $({\mathfrak {g}}_c, {\mathfrak {h}}_c)$
 is the direct sum of some copies of 
 $({\mathfrak {sl}}_n({\mathbb{C}}), {\mathfrak {gl}}_{n-1}({\mathbb{C}}))$
 $({\mathfrak {o}}_n({\mathbb{C}}), {\mathfrak {o}}_{n-1}({\mathbb{C}}))$, 
 and the trivial ones
 up to outer automorphisms.  
Therefore the real forms
 such as
$(SL(n,{\mathbb{R}}), GL(n-1,{\mathbb{R}}))$,
$(SU(p,q), U(p-1,q))$,
$(O(p,q), O(p-1,q))$
 are examples
 of the pair $(G,H)$ satisfying 
 \eqref{eqn:BB},
 and therefore \eqref{eqn:PP}, too.  
\par\noindent
{\rm{2)}}\enspace
The symmetric pair $(G,H)=(SO(n,1),SO(k) \times SO(n-k,1))$
 is an example of the pair
 that satisfies the condition \eqref{eqn:PP}
 but does not satisfy \eqref{eqn:BB}
 for $1 < k < n$.  
Likewise 
 $(G,H)=(SU(n,1),S(U(k)\times U(n-k,1)))$
 and $(Sp(n,1),Sp(k)\times Sp(n-k,1))$
 satisfy \eqref{eqn:PP} but not \eqref{eqn:BB}.  
\end{example}

Recently, 
 \lq{multiplicity-one theorems}\rq\
 have been proved in \cite{SZ},
 asserting that the upper bound \eqref{eqn:GH2}
 equals one for certain real forms $(G,H)$
 satisfying the property \eqref{eqn:BB}, 
which gives a finer result than Theorem \ref{thm:D} 1).  
However it should be noted that the uniform bound \eqref{eqn:GH2}
 can be greater than one 
 for some other real forms $(G,H)$ satisfying \eqref{eqn:BB}.
For instance,
 the upper bound \eqref{eqn:GH2}
 equals 2
 if $(G,H)=(SL(2,{\mathbb{R}}),GL(1,{\mathbb{R}})_+)$.  
Our approach here is based on 
 the theory of systems
 of partial differential equations
 with regular singularities, and 
 is completely different from \cite{A, SZ}
 which is based on the Gelfand--Kazhdan criterion.

Our approach using hyperfunction boundary value maps
 naturally connects multiplicities
 with the geometry of the real flag variety.  
As one of applications of Theorem \ref{thm:2.1}
  we can obtain the following geometric result from 
 infinite dimensional representation theory:
\begin{coralph}
\label{cor:D}
For any closed subgroup $H$ of $G$, 
 the number of open $H$-orbits on $G/P$ does not exceed 
 the order of the little Weyl group $W(\mathfrak{a})$.  
\end{coralph}

We now outline the paper.  
In Section \ref{sec:2}
 we give a quick review on \lq{hyperfunction boundary maps}\rq\
 where no assumption
 such as $K$-finiteness is required, 
 and prove a formula 
 for the upper bound of the multiplicities
 in Theorem~\ref{thm:2.1},
 which is a key step to prove the upper estimates in Theorems \ref{thm:A}
 to \ref{thm:D}.  
Conversely, 
the proof for a lower estimate of the multiplicities
 is based on a straightforward generalization
 of the construction of the Poisson transform
 for symmetric spaces.  
Theorem \ref{thm:3.1} is a stepping stone for the lower estimates
 in Theorems \ref{thm:A} and \ref{thm:C}.  
Uniform boundedness of multiplicities
 is discussed in Section~\ref{sec:4} 
 based on Theorem~\ref{thm:2.1}, 
 combined with the Borel--Weil theorem for parabolic subgroups
 and a structural result on principal series representations.
Thus we prove the first statement of Theorem \ref{thm:B}.  
The second statement of Theorems \ref{thm:B} and \ref{thm:D}
 reduces to the classical finite dimensional results.  
In Section~\ref{sec:6}
 we discuss multiplicities
 for the restriction of irreducible representations, 
 and complete the proof of Theorems \ref{thm:C} and \ref{thm:D}
 as an application of results in Sections~\ref{sec:2} and \ref{sec:4}.

\section{An upper bound of the multiplicities}
\label{sec:2}

Let $G$ be a connected real semisimple 
Lie group with finite center,
and ${\mathfrak {g}}$ its Lie algebra.  
Let $Z(\mathfrak{g})$ be the center of the enveloping algebra
$U(\mathfrak{g})$ of the complexified Lie algebra $\mathfrak{g}_c$.
Then $Z(\mathfrak{g})$ is a polynomial ring of $\operatorname{rank}{\mathfrak {g}}$
 generators,
 and the Harish-Chandra isomorphism gives a parametrization of maximal ideals of $Z(\mathfrak{g})$:
\[
\Hom_{\mathbb{C}\textrm{-alg}}(Z(\mathfrak{g}),\mathbb{C})
\simeq \mathfrak{j}_c^*/W(\mathfrak{j}), 
 \quad
 \chi_{\lambda} \longleftrightarrow \lambda, 
\]
where ${\mathfrak {j}}$ is a Cartan subalgebra of ${\mathfrak {g}}$
 and $W({\mathfrak {j}})$ is the Weyl group 
 for the root system for $({\mathfrak {g}}_c, {\mathfrak {j}}_c)$.

Let $\pi$ be a continuous representation of $G$ on a complete locally convex
vector space $V$.
Define $V^\infty$ to be the subspace of vectors $v\in V$ for which 
$g\mapsto \pi(g)v$ is a $C^\infty$ map from $G$ into $V$.
Then $V^{\infty}$ is a dense, $G$-invariant subspace of $V$.
Let $\pi^\infty$ denote the restriction of $\pi$ to $V^\infty$.
Then $\pi^\infty$ is a continuous representation on $V^\infty$
endowed with a natural Fr\'{e}chet topology, 
 and is called a \textit{smooth representation}.
It has a property that $(V^{\infty})^{\infty} = V^\infty$.
Following Harish-Chandra 
 we call $\pi$ is {\it{quasi-simple}}
 if $\pi^{\infty}$ restricts
 to scalar multiplication
 on $Z({\mathfrak {g}})$.  

We fix a maximal compact subgroup $K$ of $G$.   
We recall 
(see \cite[Chapters 3,11]{W}, for some further details):
\begin{definition}
\label{def:adm}
A continuous representation $(\pi,V)$ of $G$ of finite length is called
\textit{admissible}
if one of the following equivalent conditions
 are satisfied:
\begin{align}
&\text{$\pi^{\infty}$ is $Z({\mathfrak {g}})$ finite.}
\label{eqn:ad1}
\\
&\text{
$\dim\Hom_K(\delta,\pi)<\infty$ for any irreducible representation
$\delta$ of $K$.
}
\label{eqn:ad2}
\end{align}
\end{definition}
Then the space $V_K$ consisting of $K$-finite vectors of $V$
 is contained in $V^{\infty}$, 
 and we write $\pi_K$ for the underlying $(\mathfrak{g},K)$-module defined on $V_K$.  
We denote by $\widehat{G}_{\operatorname{ad}}$ the set of equivalence classes of
irreducible, admissible representations of $G$ on complete locally
convex topological vector spaces, 
 and by $\widehat G_{\infty}$
 that of smooth ones.  
Here two continuous representations $(\pi_1,V_1)$ and $(\pi_2,V_2)$ of $G$
are defined to be equivalent if there exists a homeomorphic $G$-homomorphism
$T:V_1\to V_2$.  
Naturally we may regard 
$
     \widehat G_{\operatorname{ad}} \supset \widehat G_{\infty} \supset \widehat G_{\operatorname{f}}.  
$

Let $H$ be a closed subgroup of $G$.  
We denote by ${\mathcal{V}}_{\tau}$ the $G$-equivariant vector bundle 
$
  G \times_H V_{\tau}
$
over $G/H$ associated with a finite dimensional representation
$(\tau,V_{\tau})$ of $H$.  
Then we have a representation $\pi$ of $G$
 on the space of sections 
\begin{equation*}
{\mathcal {F}}(G/H;\tau) \simeq{} \{f\in \mathcal{F}(G,V_\tau): f(g h) = \tau(h)^{-1} f(g) 
 \text{ for } h \in H, g \in G\},
\end{equation*}
where 
${\mathcal {F}}={\mathcal{A}}$, $C^\infty$, $\mathcal{D}'$ or ${\mathcal {B}}$
 denote the sheaves of analytic functions, 
 smooth functions, 
 distributions, 
 or hyperfunctions, 
 respectively.

For each $\lambda \in {\mathfrak {j}}_c^*$, 
 the Lie algebra ${\mathfrak {g}}$ acts on 
\begin{align}
  {\mathcal {F}}
(G/H;\tau)_{\lambda}
  \equiv&
{\mathcal {F}}(G/H; \tau, \chi_{\lambda})
\notag
\\
:=&
 \{
   f \in  {\mathcal {F}}(G/H;\tau):d \pi (D)f =\chi_{\lambda}(D)f
  \hphantom{i}\text{ for any $D \in Z(\mathfrak g)$}
\}.  
\label{eqn:FGHInd}
\end{align}
Let $E(G/H;\tau)_{\lambda}$ be 
 the subspace consisting of $K$-finite vectors, 
 which is independent of ${\mathcal {F}}$ 
 as far as $\dim \tau < \infty$
 by analytic elliptic regularity
 \cite[Theorem 3.4.4]{KKK}
 because $d\pi(C_G-2C_K)$ is an elliptic operator, 
 where $C_G$ is the Casimir element of ${\mathfrak {g}}$,
 and $C_K$ is that for ${\mathfrak {k}}$
 with the induced symmetric bilinear form from the restriction
 of the Killing form of ${\mathfrak{g}}$.

The significance of the geometric condition \eqref{eqn:HP} is summarized as follows:
\begin{theorem}
\label{thm:openHP}
If there exists an open $H$-orbit on $G/P$, 
then the $({\mathfrak {g}}, K)$-module $E(G/H;\tau)_{\lambda}$ is
 of finite length for any finite dimensional representation $\tau$ of $H$
 and any $\lambda \in {\mathfrak {j}}_c^*$.  
In particular,
 $C^{\infty}(G/H;\tau)_{\lambda}$ is an admissible representation of $G$.  

\end{theorem}

The first statement of Theorem \ref{thm:A}
 follows from Theorem \ref{thm:openHP}.  
The main goal of this section is
 to give a quantitative estimate of Theorem \ref{thm:openHP}, 
namely, 
 an upper bound for the multiplicities
 of irreducible subquotients
 in $E(G/H;\tau)_{\lambda}$
 under the condition \eqref{eqn:HP} (see Theorem \ref{thm:2.1}).  
In the course of its proof,
 we prove Theorem \ref{thm:openHP}, too.      

Let us fix some notation.  
Let ${\mathfrak {g}}={\mathfrak {k}}+{\mathfrak {s}}$ be
 the Cartan decomposition corresponding to $K$, 
 and take a Cartan subalgebra $\mathfrak{j}$ of $\mathfrak{g}$ 
 such that $\mathfrak{a}:= \mathfrak{j} \cap \mathfrak{s}$ is a maximal
 abelian subspace in $\mathfrak{s}$.
We put $\mathfrak{t} = \mathfrak{k}\cap\mathfrak{j}$.  
Let $\mathfrak{j}_c$, $\mathfrak{a}_c$ and $\mathfrak{t}_c$ be the complexifications of
$\mathfrak{j}$, $\mathfrak{a}$ and $\mathfrak{t}$, and let denote by
$\mathfrak{j}_c^*$, $\mathfrak{a}_c^*$ and $\mathfrak{t}_c^*$ the spaces of complex linear
forms on them, respectively.  
By the Killing form of $\mathfrak{g}_c$ we identify $\mathfrak{a}_c^*$ and 
$\mathfrak{t}_c^*$ with subspaces of $\mathfrak{j}_c^*$.  
Let $\Sigma(\mathfrak{j})$, $\Sigma(\mathfrak{t})$ and $\Sigma(\mathfrak{a})$ be the 
set of the roots for the pairs $(\mathfrak{g}_c, \mathfrak{j}_c)$,
$(\mathfrak{m}_c,\mathfrak{t}_c)$ and $(\mathfrak{g}, \mathfrak{a})$, respectively,
and let $W(\mathfrak{j})$, $W(\mathfrak{t})$ and $W(\mathfrak{a})$ be the
associated Weyl groups.  
Here $\mathfrak{m}_c$ is the centralizer of $\mathfrak{a}_c$ in $\mathfrak{k}_c$.
We fix compatible positive systems $\Sigma(\mathfrak{t})^+$, $\Sigma(\mathfrak{j})^+$ 
and $\Sigma(\mathfrak{a})^+$, and let $\rho$ denote half the sum of 
 roots in $\Sigma(\mathfrak{j})^+$ and we put $\rhoc = \rho|_{\mathfrak{t}}$ and 
$\rhon= \rho|_{\mathfrak{a}}$.
Naturally we have $\Sigma(\mathfrak{t})^+\subset\Sigma(\mathfrak{j})^+$.
Put $A=\exp\mathfrak{a}$ and let $M$ be the centralizer of $\mathfrak{a}$ in $K$, 
 $L:=MA$, 
 and $N$ the maximal nilpotent subgroup of $G$ corresponding to
$\Sigma(\mathfrak{a})^+$.  Then $P=LN =MAN$ is a minimal parabolic subgroup.  
We denote by $\mathbb{C}_{\rhon}$
 the one dimensional representation of $P$
 given by $p \mapsto |\det (\Ad (p):{\mathfrak {n}} \to {\mathfrak {n}})|^{\frac 1 2}$.  
Its differential representation equals $\rhon$
 when restricted to $\mathfrak{j}$.

Given $(\zeta,V_\zeta) \in \Lrep$,
 we extend it to a representation of $P$ with trivial action of $N$, 
 and define another irreducible representation of $P$ by 
\begin{equation}
\label{eqn:Vzp}
     V_{\zeta,P} := V_\zeta \otimes \mathbb{C}_\rhon.
\end{equation}
Similarly, 
 a $\bar{P}$-module 
$
     V_{\zeta,\bar{P}} := V_\zeta \otimes \mathbb{C}_\rhonbar
$
 is defined.
Let $\mathcal{V}_{\zeta,P}:=G\times_P V_{\zeta,P}$ be a
 $G$-equivariant vector bundle over $G/P$ associated with the $P$-module
 $\vzp$,
and we write $\mathcal{F}(U;\vzp)$ for the space of
 ${\mathcal{F}}={\mathcal{A}}$, ${\mathcal{B}}$, $C^\infty$ or
 $\mathcal{D}'$-valued sections of ${\mathcal{V}}_{\zeta,P}$ on an open set $U$ of $G/P$.
We write $\ipgz$ for  the underlying $({\mathfrak {g}},K)$
-module of the normalized principal series representation ${\mathcal{F}}(G/P; \vzp)$.  
Then the $Z(\mathfrak{g})$-infinitesimal character of 
    $\ipgz$
equals
 $d \zeta +\rho_{\mathfrak{t}} \in \mathfrak{j}_c^*$
 where $d \zeta$ denotes the highest weight 
 of the finite dimensional representation $\zeta$ of the Lie algebra
 ${\mathfrak {m}} + {\mathfrak {a}}$
 with respect to $\Sigma ({\mathfrak {t}})^+$.

For two continuous representations $\pi$ and $\pi'$, 
 we write $\Hom_G(\pi,\pi')$ for the space of continuous $G$-homomorphisms.  
For two $(\mathfrak{g},K)$-modules $E$, $E'$,
we set
$$
  c_{\mathfrak{g},K}(E,E')
  := \dim \Hom_{(\mathfrak{g},K)}(E, E').
$$

By a little abuse of notation,
 we also write $c_{\mathfrak{g},K}(\pi,E')$ for $c_{\mathfrak{g},K}(\pi_K,E')$
if $\pi_K$ is
 the underlying $(\mathfrak{g},K)$-module of
 $\pi \in\widehat{G}_{\operatorname{ad}}$.
We recall the following fundamental results on $({\mathfrak {g}}, K)$-modules
 and their globalizations:
\begin{lemma}
\label{lem:top}
{\rm{1)}}\enspace
For any two admissible representations $\pi$, $\pi'$ on complete, 
 locally convex vector space,
 we have
\begin{equation}
\label{eqn:Gg}
  \dim \Hom_{G}(\pi,\pi') \le c_{\mathfrak g, K}(\pi_K, \pi_K')
\end{equation}
{\rm{2)}}\enspace
For any two admissible $({\mathfrak {g}}, K)$-modules $E$, $E'$, 
 there exist admissible representations $\pi$, $\pi'$ of $G$
 such that the equality holds in \eqref{eqn:Gg} with 
$\pi_K \simeq E$ and $\pi_K' \simeq E'$.   
\end{lemma}
The first statement is easy .  
For 2), 
 the choice of such globalizations
 is not unique.  
For example,
 we can take $\pi$ and $\pi'$ to be smooth representations, 
 by the Casselman--Wallach completion 
 \cite[Theorem 11.6.7]{W}.

For $\pi \in \widehat{G}_{\operatorname{ad}}$,
$(\tau,V_\tau)\in\widehat{H}_{\operatorname{f}}$,
and 
$
      (\zeta,V_\zeta)\in\widehat{L}_{\operatorname{f}}\simeq \widehat P_{\operatorname{f}}
 \simeq \widehat{\bar P}_{\operatorname{f}}$,
 we define
\begin{align}\label{eqn:cpitau}
  c_{\mathfrak{g},K}(\pi,\Ind_H^G\tau)
  &:= \dim \Hom_{(\mathfrak{g},K)}(\pi_K, \mathcal{F}(G/H;\tau)),
\\
      c_{H\cap\bar P}(V_{\zeta}, V_{\tau,\bar{P}})
 &:= \dim \Hom_{H \cap \bar{P}}(V_{\zeta},V_{\tau,\bar{P}}).
\nonumber
\end{align}
Then, \eqref{eqn:cpitau}
 is independent of $\mathcal{F}$
 because the image of a $({\mathfrak {g}},K)$-homomorphism
 is contained in ${\mathcal {A}}(G/H;\tau)$
 by analytic elliptic regularity.  

We set
\begin{align}\label{eqn:wjlmd}
W_\lambda &:=\{w \in W(\mathfrak{j}): w \lambda = \lambda\},
\nonumber
\\
W(\mathfrak{j}; \lambda)
&      :=
      \{v \in W(\mathfrak{j}):(v \lambda)_{|\mathfrak{a}} = \lambda_{|\mathfrak{a}}\}.
\end{align}
We say $\lambda \in \mathfrak{j}_c^*$ is \textit{regular} if $W_\lambda=\{e\}$.

The first main result of this paper is:
\begin{theorem}
\label{thm:2.1}
Let $H$ be a closed subgroup of $G$ 
 and suppose $H\bar P$ is open in $G$. 
Then for any finite dimensional representation $\tau$ of $H$
 and any $\pi\in{\widehat G}_{\operatorname{ad}}$, 
 we have 
\begin{equation}
\label{eqn:2.1}
  c_{\mathfrak{g},K}(\pi,\Ind_H^G\tau)
   \le 
   \# \left(W(\mathfrak{t})\backslash W(\mathfrak{j}; \lambda)\right)
   \sum_{\zeta\in \Lrep}
    c_{\mathfrak{g},K}(\pi, \ipgz)
      \cdot c_{H\cap\bar P}(V_{\zeta}, V_{\tau,\bar P}).
\end{equation}
Here $\lambda \in {\mathfrak {j}}_c^*$ is the $Z(\mathfrak{g})$-infinitesimal character of $\pi$.  
If $\lambda$ is regular,
 then we have
$$
  c_{\mathfrak{g},K}(\pi,\Ind_H^G\tau)
   \le \sum_{\zeta\in \Lrep}
      c_{\mathfrak{g},K}(\pi, \ipgz)
      \cdot c_{H\cap\bar P}(V_{\zeta}, V_{\tau, \bar P}).
$$
\end{theorem}

Note that $c_{\mathfrak{g},K}(\pi, \ipgz)$ 
 is nonzero only for finitely many $\zeta\in\widehat{L}_{\operatorname{f}}$ for a
 fixed $\pi\in\widehat{G}_{\operatorname{ad}}$.
Further, $c_{\mathfrak{g},K}(\pi,\ipgz)$ is finite for any
 $\pi$ and $\zeta$ (in fact, it is uniformly bounded, see Proposition \ref{prop:4.2}).
Hence if $HgP$ is open for some $g\in G$, 
 then $c_{\mathfrak{g},K}(\pi,\Ind_H^G\tau)<\infty$, 
 which follows from Theorem~\ref{thm:2.1} 
 with $\bar P$ replaced by $g P g^{-1}$.

\begin{remark}
\label{rem:2.4}
{\rm 1)}
If $G$ is compact, then $G=P=\bar P$ and the equality holds in \eqref{eqn:2.1},
which is the Frobenius reciprocity theorem.

\noindent
{\rm 2)}
If $H=N$
 then the assumption in Theorem~\ref{thm:2.1} is satisfied.
In particular,
 if $\zeta_{|\mathfrak{a}} \in\mathfrak{a}_c^*$ is generic,
\eqref{eqn:2.1} implies the following inequality:
\begin{equation}
\label{eqn:Lynch}
 c_{\mathfrak{g},K}(\ipgz,\Ind_N^G\tau)\le\#W(\mathfrak{a})\cdot\dim\zeta.
\end{equation}
In this special case,
our estimate \eqref{eqn:2.1} is best possible.  
Indeed the equality holds in \eqref{eqn:Lynch}
 as was proved by T. Lynch \cite[Theorem 6.4]{Ly}.  

\noindent
{\rm 3)}
If $(G,H)$ is a semisimple symmetric pair,
 then the assumption in Theorem~\ref{thm:2.1} is satisfied.
In this special case,
 our estimate \eqref{eqn:2.1}
 improves the estimate by van den Ban \cite{Ba}
 based on a different method.  
See Example \ref{ex:AB}.

\noindent
{\rm 4)}
There exists an open $H$-orbit in $G/P$
 if and only if the number of 
the $H$-orbits in $G/P$ is finite.
The complex case was proved by Brion and Vinberg,
and the proof of the general case
 was given by the Matsuki reduction to the real rank group (\cite{M})
and the earlier classification of such subgroups $H$
 for the real rank group by Kimelfeld \cite{Kim}.  
An analogous statement does not hold if $P$ is replaced by
 general parabolic subgroups $P$.
See \cite{Bien}, 
 and also \cite[\S 2]{Kb5}
 and references therein.

\noindent
{\rm 5)}
Once we tell a priori the finiteness of the multiplicities
 of a representation $\pi$ in the induced representation 
$\Ind_H^G\tau$ by Theorem~\ref{thm:2.1},
 we may wish to understand functions that belong to a subrepresentation isomorphic to $\pi$
in $\Ind_H^G\tau$.
Some real spherical homogeneous spaces $G/H$ 
including symmetric spaces admit 
a generalized Cartan decomposition $G = K A H$
 with split abelian subgroup $A$,
 which is a useful geometric structure to analyze the asymptotic behavior
 of those functions
 by
 the reduction to $A$
 (cf. \cite[Remark 3.6]{Kb5}).  
On the other hand, 
 we obtained in 
 \cite{Kb97}
 some growth estimate at infinity
 of the $G$-invariant Radon measure
 on non-symmetric reductive spaces
 without a generalized Cartan decomposition $G = K A H$.  
\end{remark}

Our machinery for the proof of Theorem~\ref{thm:2.1}
 is the theory of regular
singularities of a system of partial differential equations \cite{KO, O1}.   
We regard the group manifold $G$
 as a symmetric space $(G \times G)/\Delta G$,
and apply the construction of taking the boundary values of $\mathcal{B}(G;\chi_{\lambda})$ 
to the hyperfunction valued principal series of $G\times G$. 
If a solution $f \in {\mathcal{B}}(G;\chi_{\lambda})$
 defined in \eqref{eqn:Blmd} below
 is ideally analytic at a boundary point of $G$
 in the compactification $\widetilde G$
 constructed in \cite[\S 1]{O2}
 then $f$ is expressed as a sum of convergent series.  
If $f$ is $(K \times K)$-finite 
then $f$ is automatically ideally analytic at
 any boundary point
 and this expression was studied earlier
 by Harish-Chandra, 
 and then by Casselman and Mili\v ci\'c
 among others
(see \cite[vol.1, Ch.4]{W}
 and references therein).  
However,
 in our setting, 
 we cannot assume 
 that $f$ is $(K \times K)$-finite.  
The advantage of our approach
 is that the boundary maps are well-defined
 inductively
 even locally 
 (without $K$-finiteness condition)
 as $({\mathfrak {g}}+ {\mathfrak {g}})$-homomorphisms, 
 which enables us to capture
 the left ${\mathfrak {g}}$-module
 ${\mathcal{B}}(G;\chi_{\lambda}, \tau)$ 
 as a filtered module 
 by assuming only the condition \eqref{eqn:HP}.   
We review briefly some results 
 of \cite{O2} 
 in a way that we need.  

We let $G\times G$ act on
the space $\mathcal{B}(G)$ of hyperfunctions on $G$ from the
left and right:
$$
    (\pi_L(g)\scirc\pi_R(g')f)(x) = f(g^{-1}xg') 
    \quad
    \text{for }
    (g,g')\in G\times G\text{ and }f\in\mathcal{B}(G).  
$$
For $\chi_\lambda \in \Hom_{\mathbb{C}\text{\upshape-alg}}(Z(\mathfrak{g}),\mathbb{C}) 
 \simeq \mathfrak{j}_c^*/W(\mathfrak{j})$
 we define
\begin{equation}
\label{eqn:Blmd}
\mathcal{B}(G; \chi_\lambda) := \{f\in\mathcal{B}(G):d \pi_L(D)f = \chi_{\lambda}(D)f
 \text{ for any } D \in Z(\mathfrak{g})\}.  
\end{equation}

The boundary value maps are defined inductively
 as $\mathfrak{g} + \mathfrak{g}$-maps as follows:
We set
\[
\Xi := \{(\lambda,\mu)\in\mathfrak{j}_c^* \times \mathfrak{a}_c^*:
         (w\lambda)|_{\mathfrak{a}} = \mu
         \quad\text{for some $w\in W(\mathfrak{j})$} \}.
\]
For $(\lambda,\mu)\in\Xi$,
we consider the following finite set of irreducible representations of
the Levi subgroup $L=MA$ defined as
\[
A(\lambda,\mu)
:= \{ \zeta\in\widehat{L}_{\operatorname{f}}: 
d\zeta+\rho_{\mathfrak{t}}\in W(\mathfrak{j})\lambda, d\zeta|_{\mathfrak{a}}
     = \mu \}.
\]
Further, 
 corresponding to the \lq{logarithmic terms}\rq, 
 we recall from \cite[Proposition~2.8]{O2}
the multiplicity function $N:\Xi\to\mathbb{N}$ with the properties
\begin{equation}
\label{eqn:2.2}
\left\{
\aligned
 &N_{\lambda,\mu} \le \#\{w\in W(\mathfrak{j}):
   (w\lambda)_{|\mathfrak{a}} = \mu\}/\#W(\mathfrak{t}), 
\\
 & N_{\lambda,\mu} \le 1
 \quad
  \text{ if $\ang{\lambda}{\alpha}\neq 0$ for any $\alpha\in \Sigma(\mathfrak{j})$},
\\
 & N_{w\lambda,\mu} = N_{\lambda,\mu}
\quad\text{for any $w\in W(\mathfrak{j})$}.
\endaligned
\right.
\end{equation}
For $\lambda\in\mathfrak{j}_c^*$ we define a finite set
\[
I_\lambda
:= \{(\mu,i): (\lambda,\mu) \in \Xi, \  i=1,\dots,N_{\lambda,\mu} \}.
\]
Clearly $I_\lambda=I_{w\lambda}$ for any $w\in W(\mathfrak{j})$.
Fix
 $Y\in\mathfrak{a}$ such that $\alpha(Y)>0$ for any $\alpha\in\Sigma(\mathfrak{a})^+$
and that 
$\nu(Y)\ne\mu(Y)$ whenever $\nu\ne\mu$ with
$(\lambda,\nu),(\lambda,\mu)\in\Xi$, 
and we give a lexicographical order $\prec$ on $I_\lambda$ by
$(\nu, j ) \prec (\mu,i)$ if and only if 
$\operatorname{Re}(\mu-\nu)(Y) + \epsilon\operatorname{Im}(\mu-\nu)(Y)
 + \epsilon^2(i-j) > 0$ for $0<\epsilon\ll 1$.

Let $U$ be an open set in $(G \times G)/(P \times \bar{P})$.
We denote by $\zeta^*$ the contragredient representation of $\zeta$.
Then we have the boundary value maps
$$
\beta_\mu^i:\ \mathcal{B}(G; \chi_{\lambda})_{\mu,i}\ \rightarrow\ 
    \bigoplus_{\zeta \in A(\lambda,\mu)}
     \mathcal{B}(U;\vzzp) 
$$
for each $(\mu,i)\in I_\lambda$ on the subspace
$\mathcal{B}(G;\chi_\lambda)_{\mu,i}$ defined inductively by
$$
  \mathcal{B}(G; \chi_{\lambda})_{\mu,i} 
  := 
\begin{cases}
\mathcal {B}(G;\chi_{\lambda})
\qquad
&\text{if $(\mu,i)$ is the smallest,}
\\
\displaystyle{\bigcap
   _{\substack{ (\nu,j)\in I_\lambda\\ (\nu, j) \prec (\mu,i)}}}
   \Ker \beta_{\nu}^j
\qquad
    &\text{otherwise}.
\end{cases}
$$
The subspaces $\mathcal{B}(G; \chi_\lambda)_{\mu,i}$ 
 with the partial order $\prec$
induces a gradation of $\mathcal{B}(G; \chi_\lambda)$,
and we write
  $\operatorname{gr} \mathcal{B}(G;\chi_\lambda)$ for the
corresponding graded module.
Then 
the induced maps
\[
\overline{\beta_\mu^i}:
\mathcal{B}(G;\chi_\lambda)_{\mu,i}/\Ker\beta_\mu^i
\to \bigoplus_{\zeta\in A(\lambda,\mu)}
\mathcal{B}(U;\vzzp)
\]
give rise to a $\mathfrak{g}\oplus\mathfrak{g}$-homomorphism
$\bar{\beta}=\oplus_{(\mu,i)\in I_\lambda}\overline{\beta_\mu^i}$ 
on the graded module:
$$
   \bar{\beta} 
    :
   \operatorname{gr} \mathcal{B}(G;\chi_\lambda)
   \to
   \bigoplus_{(\mu,i)\in I_\lambda}
   \bigoplus_{\zeta\in A(\lambda,\mu)}
   \mathcal{B}(U; \vzzp).
$$
Moreover $\bar{\beta}$ respects the action of the subgroup of
$G\times G$ that stabilizes $U$.

Assume that $H \bar P$ is open in $G$.  
We set $U:=(G\times H\bar{P})/(P\times\bar{P})$.
Then we have a $(\mathfrak{g}\times H)$-homomorphism
$$
   \bar{\beta}
   :
   \operatorname{gr} \mathcal{B}(G;\chi_\lambda)_{K\times1}
   \to
   \bigoplus_{(\mu,i)\in I_\lambda}
   \bigoplus_{\zeta\in A(\lambda,\mu)}
   \mathcal{B}(U; \vzzp).  
$$
It is important to note that 
 Holmgren's uniqueness principle
 for hyperfunctions holds,
 i.e. 
if $u\in\mathcal{B}(G; \chi_{\lambda})$ satisfies 
$\beta_\mu^i(u)=0$ for all $(\mu,i)\in I_\lambda$,
then $u$ vanishes on an open subset of $G$ 
(see \cite[\S 3]{O2}).
Therefore $\bar \beta$ is injective
 since $\mathcal{B}(G;\chi_\lambda)_{K\times1} \subset \mathcal{A}(G)$
by analytic elliptic regularity.
Passing to $1 \times \Delta(H)$-fixed vectors 
 in the $\mathfrak{g} \times H \times H$-map
$$
\bar\beta \otimes  \id : 
   \operatorname{gr} \mathcal{B}(G;\chi_\lambda)_{K\times1} \otimes V_\tau
   \to
   \bigoplus_{(\mu,i)\in I_\lambda}
 \bigoplus_{\zeta\in A(\lambda,\mu)}
   \mathcal{B}(U; \vzzp) \otimes V_\tau,
$$
we get an injective $\mathfrak{g}$-map
$$
\bar\beta : 
   (\operatorname{gr} \mathcal{B}(G;\chi_\lambda)_K \otimes V_\tau)^{\Delta(H)}
   \to
   \bigoplus_{(\mu,i)\in I_\lambda}
   \bigoplus_{\zeta\in A(\lambda,\mu)}
   \left(\mathcal{B}(U; \vzzp) 
          \otimes V_\tau\right)^{\Delta(H)}.
$$
In light of the natural isomorphism
$$
\left(\mathcal{B}(H \bar{P}/\bar{P};V_{\zeta^*, \bar P}) \otimes V_\tau\right)^{\Delta(H)}
 \simeq
 \left(V_{\zeta^*,\bar P} \otimes V_\tau\right)^{\Delta(H \cap \bar{P})}
  \simeq \Hom_{H \cap \bar P}(V_{\zeta},V_{\tau, \bar P}),
$$
we have thus
$$
   \left(\mathcal{B}(U; \vzzp) 
          \otimes V_\tau\right)^{\Delta(H)}
   \simeq
   \mathcal{B}(G/P;\vzp) \otimes
   \Hom_{H \cap \bar P}(V_{\zeta},V_{\tau, \bar P}).
$$
Hence we have obtained an injective $(\mathfrak{g},K)$-homomorphism 
\[
\bar{\beta}:
\operatorname{gr}\mathcal{B}(G/H;\tau,\chi_\lambda)_K
\to
\bigoplus_{\zeta\in A(\lambda,\mu)}
\ipgz \otimes
\mathbb{C}^{N_{\lambda,\mu}} \otimes
\Hom_{H\cap\bar{P}} (V_{\zeta},V_{\tau,\bar P}).
\]

The set of irreducible subquotients
 of the $({\mathfrak {g}},K)$-module
 ${\mathcal{B}}(G/H;\tau)_{\lambda}$ is the same 
 with that of the graded $({\mathfrak {g}}, K)$-module 
 $\operatorname{gr}{\mathcal {B}}(G/H;\tau, \chi_{\lambda})_K
\simeq \operatorname{gr}E(G/H;\tau)_{\lambda}$.  
This completes the proof of Theorem \ref{thm:openHP}.

Let $\pi \in \widehat{G}_{\operatorname{ad}}$, 
 and $\lambda$ be its infinitesimal character.    
Then
\begin{alignat*}{1}
   &\Hom_{\mathfrak{g},K}(\pi_K, \mathcal{B}(G/H; \tau))
\\
\simeq
   &\Hom_{\mathfrak{g},K}\left(\pi_K, (\mathcal{B}(G) \otimes V_\tau)^{\Delta(H)} \right)
\\
=
  &\Hom_{\mathfrak{g},K}\left(\pi_K, (\mathcal{B}(G;\chi_\lambda) \otimes V_\tau)^{\Delta(H)} \right)
\\
\subset
  &\Hom_{\mathfrak{g},K}\left(\pi_K, 
 (\operatorname{gr} \mathcal{B}(G;\chi_\lambda)_K \otimes V_\tau)^{\Delta(H)} \right)
\\
\subset
  &\bigoplus_{\zeta\in A(\lambda,\mu)}
  {\mathbb{C}}^{N_{\lambda, \mu}} \otimes \Hom_{\mathfrak{g},K}\left(\pi_K, \ipgz\right)
  \otimes
  \Hom_{H \cap \bar P}(V_{\zeta},V_{\tau, \bar P})
\end{alignat*}
and hence
\begin{alignat*}{1}
   c_{\mathfrak{g},K}(\pi, \Ind_H^G\tau)
    &\le \sum_{\zeta\in A(\lambda,\mu)}
      N_{\lambda, \mu} c_{\mathfrak{g},K}(\pi_K, \ipgz) 
     \cdot c_{H \cap \bar{P}} (V_{\zeta},V_{\tau,\bar P})
\\
    &= \sum_{\substack{ \zeta \in \Lrep}}
       N_{\lambda, d\zeta_{|\mathfrak{a}}} \, \cdot
      c_{\mathfrak{g},K}(\pi_K,\ipgz) 
     \cdot c_{H \cap \bar{P}} (V_{\zeta},V_{\tau, \bar P}).
\end{alignat*}
Now Theorem \ref{thm:2.1} follows from \eqref{eqn:2.2}.
\hfill
\qed

The proof of Theorem \ref{thm:2.1} gives an upper estimate of the
multiplicities of subquotients as well.  
Let denote by $[E:\pi]$
 the multiplicity of an irreducible $({\mathfrak {g}}, K)$-module 
 $\pi_K$
 occurring as a subquotient of a $(\mathfrak{g},K)$-module $E$.  

\begin{proposition}
Suppose that $H\bar P$ is open.
For any $\tau \in \widehat H_{\operatorname{f}}$
 and any $\pi\in\widehat G$ having $Z({\mathfrak {g}})$-infinitesimal character
 $\lambda \in {\mathfrak {j}}_c^{\ast}$, 
 we have
$$
  [E(G/H;\tau)_{\lambda}:\pi]
   \le \# \left(W(\mathfrak{t})\backslash W(\mathfrak{j}; \lambda)\right)
        \sum_{\zeta\in \Lrep}
        [\ipgz:\pi]
      \cdot c_{H\cap\bar P}(V_{\zeta}, V_{\tau,\bar{P}}).
$$
\end{proposition}

\begin{corollary}
\label{cor:2.4}
Suppose $H\bar P$ is open
 and 
 $\mu\in\mathfrak{a}_c^*$ satisfies
$
  \operatorname{Re}\ang{\mu}{\alpha}\ge 0
  \text{ for any }\alpha\in\Sigma(\mathfrak{a})^+.
$
Assume that
 $\mu + \rhoc \in \mathfrak{j}_c^*$ is regular
 with respect to $W(\mathfrak{j})$.
Then for any $\tau\in\widehat H_{\operatorname {f}}$ we have
\begin{equation}
\label{eqn:2.3}
 c_{\mathfrak{g},K}(I_P^G(\boldsymbol{1} \otimes \mu), C^\infty(G/H;\tau))\le 
 \#W(\mathfrak{a})
      \cdot c_{H\cap\bar P}(V_{\boldsymbol{1} \otimes\mu}, V_{\tau, \bar P}).
\end{equation}
\end{corollary}
\begin{proof}[Proof of Corollary \ref{cor:2.4}]
Let $\pi_K$ be
 the unique irreducible quotient
 of the spherical principal series representation $I_P^G(\boldsymbol{1} \otimes \mu)$.
Since the $K$-fixed vector in $I_P^G(\boldsymbol{1} \otimes \mu)$
 is cyclic (cf.~\cite{Ks}),
 we have
$$
 c_{\mathfrak{g},K}(I_P^G(\boldsymbol{1} \otimes \mu), C^\infty(G/H,\tau))
 \le 
 c_{\mathfrak{g},K}(\pi_K, C^\infty(G/H;\tau)).
$$
It follows from the theory of zonal spherical functions that
 $c_{\mathfrak{g},K}(\pi_K, \ipgz) \neq 0$
 (or, equivalently, $= 1$)
 only if $\zeta$ is of the form
  $\boldsymbol{1} \otimes w\mu$ for some $w\in W(\mathfrak{a})$.
Hence Corollary follows from \eqref{eqn:2.2} 
 and from the last formula in the proof of Theorem~\ref{thm:2.1}.
\end{proof}

\begin{example}
$\mu$ satisfies the regularity condition of Corollary \ref{cor:2.4},
 in the following cases:

\noindent
1)  $\mu = \rhon$.

\noindent
2) $\operatorname{Im} \mu$ is regular with respect to $W(\mathfrak{a})$.

The case 1) is clear.
Let us see the case 2).
If $w \in W(\mathfrak{j})$ satisfies $w (\rhoc + \mu) = \rhoc + \mu$,
 then we have $w \operatorname{Im} \mu = \operatorname{Im} \mu$
 by taking the projection to $\mathbb{R}$-span $\sqrt{-1} \, \Sigma(\mathfrak{j})$.
By Chevalley's theorem, 
 $w$ is contained in the subgroup generated by the reflection
 of the roots orthogonal to $\operatorname{Im}\mu$,
 that is, $w \in W(\mathfrak{t})$ by the assumption.
Now we have $\rhoc = (w (\rhoc + \mu))_{|\mathfrak{t}} = w \rhoc$,
 showing $w = 1$.
\end{example}

\section{A lower bound of the multiplicities}
\label{sec:3}

In this section we give a proof of Theorem \ref{thm:A} 2)
 and Corollary \ref{cor:D}.  
The key idea is to generalize the construction 
 of the Poisson transform 
known for symmetric spaces,
 see Theorem \ref{thm:3.1} below.

Let us recall 
 how irreducible finite dimensional
representations are realized into principal series representations.
As before,
let $P=LN$ be the Langlands decomposition of the minimal parabolic
subgroup $P$ of $G$,
 and $\mathfrak{n}$ the Lie algebra of $N$.  
Suppose $\sigma$ is an irreducible finite dimensional representation
of $G$ on a vector space $V_\sigma$.
Then, the Levi subgroup $L$ leaves
\[
V_\sigma^{\mathfrak{n}}
:= \{ v\in V_\sigma: d\sigma(X)v=0 \ \  \text{for any $X\in\mathfrak{n}$} \}
\]
invariant,
and acts irreducibly on it.
We denote by $\zeta_\sigma$ this representation of $L$.
Then $\sigma$ is the unique quotient of the principal series
representation $I_P^G(\zeta_\sigma)$,
or equivalently,
the contragredient representation $\sigma^*$ satisfies:
\begin{equation}\label{eqn:fsub}
\dim\Hom_{\mathfrak{g},K}(\sigma^*,I_P^G(\zeta_\sigma^*))=1.
\end{equation}
For $\sigma\in\widehat{G}_{\operatorname {f}}$ and $\tau\in\widehat{H}_{\operatorname {f}}$,
we set
\[
c_H(\sigma,\tau)
:= \dim\Hom_H(\sigma|_H,\tau).
\]
The following lower bound of the dimension of
$(\mathfrak{g},K)$-homomorphisms is crucial in the proof of
Theorem \ref{thm:A} 2) and Corollary \ref{cor:D}.

\begin{theorem}\label{thm:3.1}
Suppose that $H$ is a closed subgroup of $G$ and that there are $m$
disjoint $H$-invariant open sets of $G/P$.
Then
\[
c_{\mathfrak{g},K}(I_P^G(\zeta_\sigma),\Ind_H^G\tau)
\ge m c_H(\sigma,\tau)
\]
for any $\sigma\in\widehat{G}_{\operatorname {f}}$ and $\tau\in\widehat{H}_{\operatorname {f}}$.
\end{theorem}

In order to prove Theorem \ref{thm:3.1},
we construct $(\mathfrak{g},K)$-homomorphisms from a principal 
 series representation $\ipgz$ 
 into
$\Ind_H^G\tau$ by means of kernel hyperfunctions:
\begin{lemma}
\label{lem:3.2}
For any $\zeta\in\Lrep$ and $(\tau,V_\tau) \in\widehat H_{\operatorname {f}}$, we have
$$
   c_{\mathfrak{g},K}(\ipgz, \Ind_H^G\tau)
   \ge
  \dim \left(V_\tau \otimes \mathcal{B}(G/P;\vzdp)\right)^H.
$$
Here $\left(V_\tau \otimes \mathcal{B}(G/P;\vzdp)\right)^H$
denotes the space of $H$-fixed vectors of the diagonal action.
\end{lemma}
\begin{proof}
The natural $G$-invariant paring
$$ 
\langle \ , \ \rangle :
  \mathcal{A}(G/P;\vzp) \times \mathcal{B}(G/P;\vzdp) \to \mathbb{C}
$$
induces an injective $G$-homomorphism
$$
   \Psi :
    \mathcal{B}(G/P;\vzdp) \hookrightarrow
    \Hom_{G}(\mathcal{A}(G/P,\vzp),\mathcal{A}(G))
$$
by 
$\Psi(\chi)(u)(g):=\ang{\pi(g^{-1})u}{\chi}$
 for $\chi \in \mathcal{B}(G/P;\vzdp)$, $u \in \mathcal{A}(G/P,\vzp)$ and $g \in G$.
Here, 
we let $G$ act on
 $\Hom_{G}(\mathcal{A}(G/P,\vzp),\mathcal{A}(G))$ 
via the right translation on $\mathcal{A}(G)$.
Passing to the space of $\Delta(H)$-fixed vectors in the $G \times H$-map
$$
\Psi \otimes \id : 
    \mathcal{B}(G/P;\vzdp) \otimes V_\tau
    \hookrightarrow 
    \Hom_{G}(\mathcal{A}(G/P;\vzp),\mathcal{A}(G)) \otimes V_\tau,
$$
we have an injective map 
\begin{equation}
\label{eqn:Poisson}
     \mathcal{P}:
  \left(V_\tau \otimes \mathcal{B}(G/P;\vzdp)\right)^H \hookrightarrow
   \Hom_{G}({\mathcal{A}}(G/P;\vzp), \Ind_H^G\tau).  
\end{equation}
Hence we have proved Lemma \ref{lem:3.2}.  
\end{proof}

\begin{example}\label{ex:Poisson}
If $H=K$ and $\tau$ is the one dimensional trivial representation,
then the following three conditions
 are equivalent:
\begin{enumerate}
\item[i)]
$(V_\tau\otimes\mathcal{B}(G/P;\vzdp))^K \ne 0$, 
\item[ii)]
$\dim_{\mathbb{C}}
\left(V_\tau \otimes \mathcal{B}(G/P;\vzdp)\right)^K=1$, 
\item[iii)]
 $\zeta|_M$ is trivial.  
\end{enumerate}
The corresponding intertwining operator 
 (see \eqref{eqn:Poisson}) from ${\mathcal {A}}(G/P;\vzdp)$ into
${\mathcal{A}}(G/K)$ coincides with the Poisson transform for the Riemannian symmetric
space $G/K$ up to a scalar multiple. 
\end{example}

\begin{proof}[Proof of Theorem \ref{thm:3.1}]
Let $U_i$ $(i=1,2,\dots,m)$ be disjoint $H$-invariant open subsets of
$G/P$. 
We define $\chi_i\in\mathcal{B}(G/P)$
 by 
$$
  \chi_i(g) = \begin{cases} 1\quad\text{if }g\in U_i,\\ 
                     0\quad\text{if }g\notin U_i.
              \end{cases}
$$
Clearly,
 $\chi_i \in \mathcal{B}(G/P)^H \, (i=1,2, \dots, m)$ are linearly independent.

Next we identify $V_\sigma^*$ with the unique subspace of the
principal series representation $I_P^G(\zeta_\sigma^*)$ 
(see \eqref{eqn:fsub}).
Take linearly independent $H$-fixed elements $u_1,\ldots,u_n$ of 
 $V_\tau\otimes V_\sigma^*$
 with $n := c_H(\sigma,\tau)$,
 where we have regarded as 
 $ u_j \in \left(V_\tau \otimes \mathcal{B}(G/P;V_{\zeta_\sigma^*, P})\right)^H$.
Then $\chi_i u_j \in \left(V_\tau \otimes \mathcal{B}(G/P;V_{\zeta_{\sigma}^*, P}
)\right)^H$ 
 are well-defined and linearly independent for  $i=1,\ldots,m$ and $j=1,\ldots,n$
 because $u_j$ are real analytic.
Owing to Lemma \ref{lem:3.2},
Theorem \ref{thm:3.1} has been now proved.
\end{proof}

We pin down special cases of Theorem \ref{thm:3.1}:
\begin{example}\label{ex:lowbd}
Suppose $H$ is a closed subgroup of $G$.
\begin{enumerate}[\upshape1)]
\item  
For any $\sigma\in\widehat{G}_{\operatorname {f}}$ and $\tau\in\widehat{H}_{\operatorname {f}}$,
$
c_{\mathfrak{g},K}(I_P^G(\zeta_\sigma),\Ind_H^G\tau)
\ge c_H(\sigma,\tau).
$
\item  
Suppose that there exists $m$
disjoint $H$-invariant open sets of $G/P$.
Then
 $c_{\mathfrak{g},K}(I_P^G(\boldsymbol{1}),C^\infty(G/H))\ge m$.
\end{enumerate}
The first statement is a special case of Theorem~\ref{thm:3.1}
 by regarding $G/P$  as an (obvious) open
$H$-invariant subset,
and the second statement  corresponds to $\sigma = \boldsymbol{1}$, 
 $\tau = \boldsymbol{1}$ and $\zeta_{\sigma} = \boldsymbol{1}$.
\end{example}

Finally, we use the following elementary result for algebraic groups.
\begin{lemma}\label{lem:algH}
Suppose $H$ is an algebraic subgroup of a real algebraic semisimple
Lie group $G$.
If there is no open $H$-orbit on $G/P$,
then there exist infinitely many, disjoint $H$-invariant open sets of
$G/P$. 
\end{lemma}

For the sake of completeness,
 we give a proof of Lemma \ref{lem:algH}
 in Appendix.  

\begin{proof}[Proof of Theorem~\ref{thm:A} 2)]
Suppose there is no open $H$-orbit on $G/P$.  
Then we can take infinitely many disjoint $H$-invariant open subsets $U_i$ of
 $G/P$
 by Lemma \ref{lem:algH}.  
For a given algebraic representation $\tau\in\widehat H_{\operatorname {f}}$,
 we can take $\sigma\in\widehat G_{\operatorname {f}}$
 with $c_H(\sigma,\tau)>0$
 by the Frobenius reciprocity.  
Then  $c_{\mathfrak{g},K}(I_P^G(\zeta_\sigma), \Ind_H^G\tau) 
 =\infty$ by Theorem \ref{thm:3.1}.
Since there are at most finitely many irreducible $(\mathfrak{g},K)$-modules
 occurring in the principal series representation $I_P^G(\zeta_\sigma)$
 as subquotients, Theorem A 2) now follows.
\end{proof}

\begin{proof}[Proof of Corollary~\ref{cor:D}]
Let $m$ be the number of open $H$-orbits on $G/P$. 
By Example \ref{ex:lowbd}, 
 we have 
\[
    c_{{\mathfrak {g}}, K}(I_P^G({\bf{1}}), \Ind _H^G {\bf{1}}) \ge m.  
\]
Comparing this with Corollary~\ref{cor:2.4} in the case $\mu=\rhon$
 and $\tau= \boldsymbol{1}$, 
 we get $m \le \# W({\mathfrak {a}})$.    
\end{proof}

We end this section with a counterexample to an analogous
multiplicity-finite statement
 without algebraic assumptions in Theorem~\ref{thm:3.1}.
\begin{example}
\label{ex:alg}
Let $G = SL(2,\mathbb{R})\times\cdots\times SL(2,\mathbb{R})$ be the direct
product group of $(n+1)$-copies of $SL(2,\mathbb{R})$. 
Fix real numbers $\lambda_1,\ldots,\lambda_n$ which are
 linearly independent over $\mathbb{Q}$.
Writing
$
 k_\theta := \begin{pmatrix} \cos\theta && -\sin\theta\\ 
                      \sin\theta && \cos\theta 
             \end{pmatrix}
$
and
$
p_{t,x} := \begin{pmatrix} e^t & x \\ 0 & e^{-t} \end{pmatrix}
$,
 we define a two-dimensional subgroup of $G$ by 
$$ 
        H = \{g_{t,x} = ( p_{t,x},
       k_{\lambda_1t},\cdots,k_{\lambda_nt}): (t,x)\in\mathbb{R}^2\}.
$$
Then there is no open $H$-orbit on $G/P$ if $n>1$
 because
 $\dim G/P = n+1 > \dim H=2$. 
However,
 we still have a finite multiplicity statement:
\begin{equation}
\label{eqn:SL2}
c_{\mathfrak{g},K}(\pi,\Ind_H^G\tau) \le 2
\quad\text{ for any } \pi\in\widehat G_{\operatorname {ad}}
 \text{ and for any }\tau \in \widehat H_{\operatorname {f}}.  
\end{equation}
Let us prove \eqref{eqn:SL2}.  
We observe that any finite dimensional irreducible representation of $H$ 
factors through the quotient group
$H/[H,H] \simeq \mathbb{R}$,
and is of the form 
$
  \tau_\mu(g_{t,x}) = e^{\mu t}
$
for some $\mu\in\mathbb{C}$.  
Let 
$
    \chi_m(k_\theta):=e^{2\pi\sqrt{-1}m\theta}
$
 and 
$
    \sigma_\mu(p_{t,x}):=e^{\mu t}.
$
Then $\chi_m$ $(m\in\mathbb{Z})$ and
$\sigma_\mu$ $(\mu\in\mathbb{C})$ are one dimensional representations
of $SO(2)$ and 
$AN = \{ p_{t,x}: t,x \in \mathbb{R} \}$, 
 respectively.

For $m=(m_1,\dots,m_n)\in\mathbb{Z}^n$ and
$u\in C^\infty (G/H; \tau_\mu)$,
we define
\begin{align*}
 & (S_{m}u)(g_0,g_1,\ldots,g_m) :=
\\ 
 &\int_0^{2\pi}\cdots\int_0^{2\pi}
  u(g_0,g_1k_{\theta_1},\ldots,g_nk_{\theta_n})
  e^{2\pi\sqrt{-1}(m_1\theta_1+\cdots+ m_n\theta_n)}
\frac{d\theta_1}{2\pi}\cdots\frac{d\theta_n}{2\pi}.
\end{align*}
Then, for $t,x,\varphi_1,\dots,\varphi_n \in \mathbb{R}/2 \pi {\mathbb{Z}}$,
and $g=(g_0, g_1, \cdots, g_n)$, 
 we have
\begin{equation*}
 (S_m u) (g  g_{t,x}) 
  ={} \sigma_{\mu-2\pi\sqrt{-1}(\lambda_1m_1+\dots+\lambda_nm_n)}
       (p_{t,x}^{-1}) \prod_{j=1}^n \chi_{m_j} (k_{\varphi_j}^{-1})
       (S_m u) (g).
\end{equation*}
Thus, $S_m u$ defines an element of 
$
C^\infty (G/\widetilde{H};
 \sigma_{\mu-2\pi\sqrt{-1}\langle\lambda,m\rangle} \otimes
 \chi_{m})
$
where
$
     \chi_m=\chi_{m_1} \otimes \cdots \otimes \chi_{m_n},
$ 
$\langle\lambda,m\rangle := \lambda_1 m_1 + \dots + \lambda_n m_n,
$ 
 and 
\[
  \widetilde{H} := AN \times SO(2) \times \cdots \times SO(2).  
\]
Clearly,
$S:= \bigoplus_{m\in\mathbb{Z}^n} S_m$ gives an injective $G$-homomorphism:
$$
 S: C^\infty(G/H;\tau_\mu) \to
 \bigoplus_{m\in\mathbb{Z}^n}
 C^\infty(G/\widetilde{H};
 \sigma_{\mu-2\pi\sqrt{-1}\langle\lambda,m\rangle }
 \otimes \chi_{m} ).  
$$
Now \eqref{eqn:SL2} follows from the well-known facts on
$G_1 = SL(2,\mathbb{R})$:
\begin{enumerate}
    \renewcommand{\labelenumi}{\theenumi)}
\item  
$\# \{ \mu\in\mathbb{C}: \Hom_{(\mathfrak{g}_1,K_1)}
 (\pi_1, \Ind_{AN}^{G_1} \sigma_\mu) \ne 0 \} \le 2$, 
 for any $\pi_1 \in \widehat{G_1}_{\operatorname {ad}}$.  
\item  
$\Ind_{K_1}^{G_1} \chi_l$ is multiplicity-free for any $l\in\mathbb{Z}$.
\end{enumerate}
\end{example}

\section{Uniform boundedness of the multiplicities}
\label{sec:4}

This section is devoted to the proof of Theorem~\ref{thm:B}.
We will prove \eqref{eqn:HB} $\Rightarrow$ \eqref{eqn:HG1}
 based on the general formula
\eqref{eqn:2.1} on upper bounds of multiplicities
(see Theorem~\ref{thm:2.1}).  
The opposite implication \eqref{eqn:HG2} $\Rightarrow$ \eqref{eqn:BB} 
(or \eqref{eqn:HG3} $\Rightarrow$ \eqref{eqn:BB} when $H$ is reductive)
is proved by using Theorem \ref{thm:3.1} on lower bounds.

We begin with the following uniform estimate
 of 
 multiplicities of irreducible representations occurring 
 in principal series representations
 as subquotients 
 for which there is,
 to our knowledge,
 no direct proof in the literature.
So we will give its proof in the appendix (see Section~\ref{subsec:7.2}).

\begin{proposition}
\label{prop:4.2}
There exits a constant $N$ depending only on $G$ such that
$$
   [\ipgz:\pi] \le N
 \quad\text{for any $\pi \in \widehat G_{\operatorname {ad}}$ and for any $\zeta \in \Lrep$}.
$$
\end{proposition}

Retain the notation of Section \ref{sec:2}.  
In particular, 
 $B$ is the Borel subgroup of $G_c$ with the Lie algebra $\mathfrak{b}$
  given by the positive system $\Sigma(\mathfrak{j})^+$.
Then $\mathfrak{b}$ is contained in the complexified Lie algebra 
 ${\mathfrak {p}}_c$ of the minimal parabolic subgroup $P=LN$ of $G$.
\begin{lemma}
\label{lem:HPB}
If $H_c$ acts on $G_c/B$ with an open orbit,
then there exists 
$g\in G$ such that $H_cgB$ is open in $G_c$ and that $HgP$ is open in
$G$.
In particular, 
\eqref{eqn:HB} $\Rightarrow$ \eqref{eqn:HP}.  
\end{lemma}

\begin{proof}
Put $G_c'=\{g\in G_c: \Ad(g)\mathfrak{h}_c + \mathfrak{b} \ne \mathfrak{g}_c\}$.
Then $G_c'$ is a proper closed analytic subset of the complex manifold $G_c$.
Hence $G\not\subset G_c'$ and there exists $g\in G$ with 
$\Ad(g)\mathfrak{h}_c + \mathfrak{b} = \mathfrak{g}_c$,
which implies $\Ad(g)\mathfrak{h} + \Lie(P) = \mathfrak{g}$.
\end{proof}

Suppose that $H_c$ has an open orbit
 on $G_c/B$.  
Replacing $H$ by $g^{-1} Hg$ in Lemma \ref{lem:HPB}, 
 we may assume
 that $H_c B$ is open in $G_c$
 and $HP$ is open in $G$.  
Then we apply Theorem \ref{thm:2.1}
 and Proposition \ref{prop:4.2}
 with $\bar P$ replaced by $P$.  
Thus we have shown 
\[
  c_{\mathfrak {g}, K}(\pi, \Ind_H^G \tau)
  \le
  N \#(W({\mathfrak {t}})\backslash W({\mathfrak {j}};d \pi))
  \sum_{\substack{\zeta \in \Lrep \\ d \zeta = d \pi}} c_{H \cap P (V_{\zeta}, V_{\tau, P})}
\] 
for any $\pi \in \widehat G_{\operatorname{ad}}$
 with infinitesimal character $d \pi$
 and for any $(\tau, V_{\tau}) \in \widehat H_{\operatorname{f}}$.  
Now the implication \eqref{eqn:HB} $\Rightarrow$ \eqref{eqn:HG1}
 in Theorem \ref{thm:B} follows from Proposition \ref{prop:4.1}
 below on finite dimensional representations.  

Let $P_0$ be the identity component of $P$,
 $J$ the Cartan subgroup of $G$ with Lie algebra $\mathfrak{j}$, 
 and $Z(G)$ the center of $G$.
Let $D$ be the maximal dimension of the irreducible representations of $J$.
Note that $D \le \#(J/Z(G)\exp\mathfrak{j})=\#(P/Z(G)P_0)$
  and $D = 1$ if $G$ is linear.

\begin{proposition}
\label{prop:4.1}
Assume that $H_c B$ is open in $G_c$ and that $H P$ is open in $G$. 
For any $(\tau, V_\tau) \in \widehat{H}_{\operatorname {f}}$ and $(\zeta, V_\zeta) \in \Lrep$
 we have
 $c_{H\cap P}(V_{\zeta}, V_{\tau,P}) \le D\cdot\dim\tau$.
\end{proposition}
\begin{proof}
It follows from
 $\mathfrak{g}_c = \mathfrak{h}_c + \mathfrak{b}$ and $\mathfrak{b} \subset\mathfrak{p}_c$
 that
\begin{equation}
\label{eqn:4.1}
 \mathfrak{p}_c = (\mathfrak{h}_c \cap \mathfrak{p}_c) + \mathfrak{b}.
\end{equation}
Let $\widetilde{P}_c$ be
the connected and simply connected complex Lie group
with Lie algebra $\mathfrak{p}_c$.  
We write $(H\cap P)_c$
 and $\widetilde B$ for the connected subgroups of
$\widetilde P_c$ with Lie algebra $\mathfrak{h}_c\cap\mathfrak{p}_c$
 and $\mathfrak{b}$, 
respectively. 
Then the $P$-module ${\zeta}^*$
 uniquely corresponds to irreducible representations
 $\zeta_1$ of $J$ and $\zeta_o$ of $P_0$ by the natural map 
 $J\times P_o\ni(j,p)\mapsto jp\in P$
 and hence ${\zeta}^*$ is isomorphic
to the direct sum of $\dim\zeta_1$ copies of
$\mathcal{O}(\widetilde P_c/\widetilde B,\mathcal{L}_\lambda)$
as $\mathfrak{p}_c$-modules.
Here 
 $\mathcal{L}_\lambda$ is the $\widetilde P_c$-homogeneous holomorphic
 line bundle over $\widetilde P_c/\widetilde B$
 associated with a suitable character $\lambda$ of $B$ 
 such that the space of global holomorphic sections,
 denoted by $\mathcal{O}(\widetilde P_c/\widetilde B,\mathcal{L}_\lambda)$, 
corresponds to the Borel--Weil realization of $\zeta_o$.
Note that $\dim\zeta_1\le D$.
Passing to the space of fixed vectors under
 the diagonal action of $H \cap P$ on $V_{\tau,P}\otimes V_\zeta$,
 we have
$$
 (V_{\tau,P}\otimes V_\zeta)^{H\cap P}
  \subset \bigoplus^{\dim\zeta_1}
          \Big(V_{\tau,P}\otimes\mathcal{O}(\widetilde P_c/\widetilde B,\mathcal{L}_\lambda)
          \Big)^{\mathfrak{h}_c\cap\mathfrak{p}_c}.  
$$
Since $(H\cap P)_c$ acts on $\widetilde P_c/ \widetilde B$ with an open orbit
 by \eqref{eqn:4.1},  
 $\dim (V_{\tau, P}\otimes V_\zeta)^{H\cap P} \le \dim\zeta_1\cdot\dim\tau$
because a holomorphic function on a connected complex manifold 
 is uniquely determined by its restriction to an open subset.
Hence Proposition \ref{prop:4.1} is proved.
\end{proof}

Thus we have completed the proof of the implication \eqref{eqn:HB} $\Rightarrow$ \eqref{eqn:HG1} 
in Theorem~\ref{thm:B}.  

\begin{remark}
Let $G$ be real algebraic (not necessarily reductive),
 $\sigma$ an involution and $H = G^\sigma$.
R.~Lipsman proved that the multiplicity of the abstract Plancherel formula
 for $G/H$ is uniformly bounded 
under the hypothesis that this statement
 is true in the reductive case (\cite[Theorem~7.3]{Li}).
Theorem~\ref{thm:B} shows that his hypothesis is true
 because there always exists an open $H_c$-orbit on $G_c/B$
for any complex reductive symmetric pair $(G_c,H_c)$.
\end{remark}

Let us prove the remaining implication
 in Theorem \ref{thm:B}, 
 namely, 
 \eqref{eqn:HG2} $\Rightarrow$ \eqref{eqn:HB}
(or \eqref{eqn:HG3} $\Rightarrow$ \eqref{eqn:HB} when $H$ is reductive).  

Let $N$ be the constant in Proposition \ref{prop:4.2}.  
Then,
 for any $\pi \in \widehat G_{\operatorname{ad}}$, 
 $\zeta \in \widehat L_{\operatorname{f}}$, 
 and $\tau \in \widehat H_{\operatorname{f}}$, 
 we have 
\[
     c_{{\mathfrak {g}}, K} (\ipgz, \Ind_H^G \tau)
    \le 
     N c_{{\mathfrak {g}}, K} (\pi, \Ind_H^G \tau).  
\]
Therefore the conditions \eqref{eqn:HG2} and \eqref{eqn:HG3}
 imply
\begin{align*}
\sup_{\substack{\tau \in \widehat H_{\operatorname{f}}
      \\ \dim \tau =1}}
\sup_{\zeta \in \widehat L_{\operatorname{f}}} 
      c_{{\mathfrak {g}}, K}(\ipgz, \Ind_H^G \tau)<&\infty, 
\\
  \sup_{\zeta \in \widehat L_{\operatorname{f}}} 
      c_{{\mathfrak {g}}, K}(\ipgz, C^{\infty}(G/H))<&\infty, 
\end{align*}
respectively.  
Applying Theorem \ref{thm:3.1}
 with $m=1$, 
 we get 
\begin{align*}
\sup_{\substack{\tau \in \widehat H_{\operatorname{f}}
      \\ \dim \tau =1}}
\sup_{\sigma \in \widehat G_{\operatorname{f}}} 
      c_{H}(\sigma, \tau)<&\infty, 
\\
  \sup_{\sigma \in \widehat G_{\operatorname{f}}} 
      c_{H}(\sigma, {\bf{1}})<&\infty, 
\end{align*}
respectively.  
Hence, 
 the implication \eqref{eqn:HG2} $\Rightarrow$ \eqref{eqn:HB}
 (or \eqref{eqn:HG3} $\Rightarrow$ \eqref{eqn:HB}
 when $H$ is reductive)
 reduces to the implication (iii)$'$ $\Rightarrow$ (i)
 (or (iv)$'$ $\Rightarrow$ (i))
 in the following classical results on finite dimensional
representations: 
\begin{lemma}
[\cite{VK}]
\label{lem:HB}
Let $H_c$ be an algebraic subgroup of a complex semisimple Lie group
$G_c$. 
In what follows $\widehat{G}_{\text{\upshape alg}}$,
$\widehat{H}_{\text{\upshape alg}}$ denote the set of irreducible algebraic finite
dimensional irreducible representations of $G_c$, $H_c$, respectively.
Then the following five conditions on the pair $(G_c,H_c)$ are
equivalent:
\begin{enumerate}
\item[\upshape(i)]  
There exists an open $H_c$-orbit on $G_c/B$.
\item[\upshape (ii)]  
$c_H(\sigma,\tau)\le\dim\tau$ for any
$\sigma\in\widehat{G}_{\text{\upshape alg}}$
and $\tau\in\widehat{H}_{\text{\upshape alg}}$.
\item[\upshape(ii)$'$] 
$
\underset {\tau \in \widehat H_{\operatorname {alg}}} \sup \underset {\sigma \in \widehat G_{\operatorname {alg}}} \sup \dfrac 1 {\dim \tau} c_{H}(\sigma, \tau)<\infty.  
$

\item  [\upshape (iii)]
$c_H(\sigma,\tau)\le1$ for any 
$\sigma\in\widehat{G}_{\text{\upshape alg}}$
and $\tau\in\widehat{H}_{\text{\upshape alg}}$
such that $\dim\tau=1$.
\item[\upshape(iii)$'$] 
$
\displaystyle{\sup_{\substack{\tau \in \widehat H_{\text{\upshape alg}}\\ \dim \tau =1}}}\sup_{\sigma \in \widehat G_{\text{\upshape alg}}} c_{H}(\sigma, \tau)<\infty.  
$
\end{enumerate}
Furthermore, if $H$ is reductive, then they are also equivalent to:
\begin{itemize}
\item[\upshape(iv)]
$c_H(\sigma,\mathbf{1})\le1$ for any 
$\sigma\in\widehat{G}_{\text{\upshape alg}}$.
\item[\upshape(iv)$'$]
$
  \displaystyle{\sup_{\sigma \in \widehat G_{\text{\upshape alg}}}} c_H(\sigma, {\bf{1}})< \infty.  
$
\end{itemize}
\end{lemma}

\begin{proof}
The following implications are obvious:
\begin{alignat*}{5}
&\text{(ii)}\,\, &&\Rightarrow\,\, && \text{(iii)} \,\,&& \Rightarrow \,\, &&\text{(iv)}
\\
&\,\Downarrow &&  && \,\Downarrow && && \,\Downarrow
\\
&\text{(ii)}' \,\,&&\Rightarrow \,\,&&\text{(iii)}'\,\,&& \Rightarrow \,\,&&\text{(iv)}' 
\end{alignat*}
The implication (i) $\Rightarrow$ (ii) follows easily from the
Borel--Weil theorem.  
The non-trivial part is (iii) $\Rightarrow$ (i)
 (or (iv) $\Rightarrow$ (i)), 
 which was proved in Vinberg--Kimelfeld \cite {VK}.

Let us show the remaining (and easy)
 implication (iii)$'$ $\Rightarrow$ (iii) 
 (or (iv)$'$ $\Rightarrow$ (iv)).  
Suppose 
 $c_H(\sigma, \tau) \ge 2$
 for some $\sigma \in \widehat G_{\operatorname{alg}}$
 and $\tau \in \widehat H_{\operatorname{alg}}$
 with $\dim \tau =1$.  
Then we can find two linearly independent
 highest weight vectors $f_1$, $f_2 \in {\mathcal {O}}(G_c)$
 such that 
$
    f_j (b^{-1}gh)= \chi_{\sigma} (b) \tau (h^{-1}) f_j (g)
$
  ($j=1,2$)
 for any $b \in B$, $h \in H_c$, 
and $g \in G_c$
 where $\chi_{\sigma}$ corresponds to a highest weight of $\sigma$.  
We claim that $f_1^{i} f_2^{N-i}$
 ($0 \le i \le N$)
 are linearly independent.  
Indeed,
 suppose $a_0 f_1^N + a_1 f_1^{N-1} f_2 + \cdots + a_N f_2^N=0$
 is a linear dependence.  
Let $\lambda$ be a zero 
 of the equation $a_0 t^N + a_1 t^{N-1} + \cdots +a_N=0$.  
Since the ring ${\mathcal{O}}(G_c)$
 has no divisor,
 we have $f_1 - \lambda f_2 =0$,
 which contradicts to the linear independence
 of $f_1$ and $f_2$.  
Therefore, 
 we have $c_H(\sigma_N, \tau^N) \ge N+1$
 where $\sigma_N \in \widehat G_{\operatorname{alg}}$
 is defined to have a highest weight $\chi_{\sigma}^{N}$.  
Hence (iii)$'$ $\Rightarrow$ (iii) is shown.  
The implication (iv)$'$ $\Rightarrow$ (iv) is 
immediate by putting
 $\tau = {\bf{1}}$.  
\end{proof}

We have thus completed the proof of Theorem \ref{thm:B}.

\section{Restriction of irreducible  representations}
\label{sec:6}

In this section we discuss the restriction
 of an admissible irreducible
representation $\pi$ of a semisimple Lie group with respect to a
reductive subgroup $H$, 
and give a proof of Theorems \ref{thm:C}
  and \ref{thm:D}
 on geometric criteria for finiteness
 and boundedness
 of the dimension of 
$\Hom_H(\pi|_H,\tau)$,
the space of continuous $H$-homomorphisms
for $\tau\in\widehat{H}_{\operatorname {ad}}$.

In dealing with the restrictions
 of admissible representations
 which are not necessarily unitary, 
 we work mostly in the framework
 of smooth representations.
We begin with an elementary observation:
\begin{lemma}\label{lem:adsmooth}
Suppose $(\pi,V_\pi)\in\widehat{G}_{\operatorname {ad}}$ and
$(\tau,V_\tau)\in\widehat{H}_{\operatorname {ad}}$.
Then we have a natural injective map
$$
 \Hom_H(V_\pi,V_\tau) \to \Hom_H(V_\pi^\infty,V_\tau^\infty),
 \quad  \varphi \mapsto \varphi|_{V_\pi^\infty}.
$$
\end{lemma}

\begin{proof}
Let $\varphi: V_\pi \to V_\tau$ be a continuous $H$-homomorphism.
If $v$ is a smooth vector of $V_\pi$ as a representation of $G$,
then $v$ is a smooth vector for the representation $\pi|_H$ of the subgroup $H$,
and consequently, so is $\varphi(v)$ for $\tau$.
Since $V_\pi^\infty$ is dense in $V_\pi$,
$\varphi \mapsto \varphi|_{V_\pi^\infty}$ is injective.
\end{proof}

Let $\Delta H$ denote the diagonal subgroup 
 $\{(h,h): h \in H\}$
 in $G \times H$.  
The next lemma reduces the problem of the restriction to a problem on the induced
representation for
which we have already solved in Sections~\ref{sec:2} and \ref{sec:3}:
\begin{lemma}
\label{lem:smooth}
For any $\pi\in\widehat{G}_\infty$ and $\tau\in\widehat{H}_\infty$,
there is a natural bijection
\[
\Hom_H(\pi|_H,\tau)
\simeq
\Hom_{G\times H} (\pi\times \sigma,C^\infty(G\times H/\Delta H)).
\]
Here, 
$\tau^{*}$ is the contragredient representation
 of $H$ on the continuous dual
 (the space of distribution vectors), 
and $\sigma$ denotes its smooth representation $(\tau^*)^\infty$.  
\end{lemma}

\begin{proof}
We write $V_\pi$, $V_\tau$, and $V_{\sigma}$ for the
representation spaces of the smooth representations $\pi$, $\tau$, and 
$\sigma$, respectively.

Suppose $\varphi: V_\pi\to V_\tau$ is a continuous
$H$-homomorphism.
We define a continuous map 
$\Phi: V_\pi\times V_{\sigma} \times G\times H \to \mathbb{C}$
by
\[
\Phi(v,u;g,h) := \langle\varphi(\pi(g^{-1})v),\sigma(h^{-1})u\rangle.
\]
Then the induced map
$(v,u) \mapsto \Phi(v,u;\cdot,\cdot)$
gives a continuous $(G\times H)$-homomorphism from
$V_\pi\times V_{\sigma}$ to
$C^\infty(G\times H/\Delta H)$.

Conversely, suppose
$\Psi: V_\pi \times V_\sigma \to C^\infty(G\times H/\Delta H)$
is a continuous $(G\times H)$-homomorphism.
Then the linear map
\[
\psi: V_\pi \to V_\sigma^*,
\ 
v \mapsto \Psi(v,\cdot)(e,e)
\]
is a continuous $H$-homomorphism,
and therefore,
its image is contained in the subspace
$(W_{\sigma}^*)^\infty$ of $W_\sigma^*$.
Since every smooth representation $\tau$ of $H$
is reflexive, i.e.
 $(W_{\sigma}^*)^{\infty} \simeq V_{\tau}$,
we have now shown Lemma.
\end{proof}
Combining Lemma \ref{lem:adsmooth}
 with Lemma \ref{lem:smooth}, 
 we get
\begin{align}
\dim \Hom_H (\pi|_H, \tau) \le& \dim \Hom_H(\pi^{\infty}|_H, \tau^{\infty})
\notag
\\
=& \dim \Hom_{G \times H}(\pi^{\infty}\times (\tau^*)^{\infty}, C^{\infty}(G \times H/\Delta H)).  
\label{eqn:GHH}
\end{align}

Now Theorems \ref{thm:C} and \ref{thm:D} follow from \eqref{eqn:GHH}
 and Theorems \ref{thm:A} and \ref{thm:B}
 in light of the following elementary observation:  

\begin{lemma}
\label{lem:PPHP}
{\rm{1)}}\enspace
The condition \eqref{eqn:PP} holds for the pair $(G,H)$
 if and only if the condition \eqref{eqn:HP} holds
 for $(G \times H, \Delta H)$.  
\par\noindent
{\rm{2)}}\enspace
The condition \eqref{eqn:BB} holds for $(G,H)$
 if and only if the condition \eqref{eqn:HB}
 holds for $(G \times H, \Delta H)$.  
\end{lemma}
\begin{proof}
{\rm{1)}}\enspace
$P \times P_H$ is a minimal parabolic subgroup of $G \times H$.  
The claim follows from the natural bijection
$
  (P \times P_H) \backslash (G \times H)/ \Delta H \simeq P \backslash G/P_H.  
$
{\rm{2)}}\enspace
Similarly,
 $B \times B_H$ is a Borel subgroup
 of $G_c \times H_c$, 
 and the claim follows from the bijection
$
     (B \times B_H) \backslash (G_c \times H_c) / \Delta H_c
     \simeq
     B \backslash G_c / B_H.
$  
\end{proof}

In the case where $\pi$ is unitary, 
we can decompose the restriction $\pi|_H$ into the direct integral of
irreducible unitary representations of $H$,
and such a decomposition (\textit{branching law}) is unique
as $H$ is of type I in the sense of von Neumann algebras.
We denote by $\widehat{G}$ the set of (unitary) equivalence classes of
 irreducible unitary representations of $G$.  
For $(\pi, V_\pi) \in \widehat{G}$, $(\tau, W_\tau) \in \widehat{H}$, 
$
  \varphi \in \Hom_{H} (\tau, \pi|_H)$
 gives an irreducible summand
 $\varphi (W_\tau)$ in $V_{\pi}$. 
 
As an immediate corollary of Theorems \ref{thm:C} and \ref{thm:D},
we give an upper bound of the multiplicity
 in the discrete part:
\begin{theorem}
\label{thm:6.1}
Suppose $(G,H)$ is a pair of reductive Lie groups.

\noindent
{{\rm 1)}}
 If there is an open $P_H$-orbit on $G/P$,
 then
 $\dim \Hom_H(\tau, \pi|_{H})<\infty$ 
 for any $\pi \in \widehat{G}$ and $\tau \in \widehat{H}$.

\noindent
{{\rm 2)}}
 If there is an open $B_H$-orbit on $G_c/B$,
 $\underset{\pi \in \widehat{G}, \tau \in \widehat{H}}\sup
 \dim \Hom_H(\tau, \pi|_{H})<\infty$. 
\end{theorem}
\begin{proof}
[Proof of Theorem \ref{thm:6.1}]
Since the adjoint map gives an anti-linear bijection 
\[
  \Hom_H(V_{\tau}, V_{\pi}) \simeq \Hom_H(V_{\pi}, V_{\tau}), 
\]
Theorem \ref{thm:6.1} follows from Theorems \ref{thm:B}
 and \ref{thm:D}.  
\end{proof}
\begin{remark}
{\rm 1)}
Theorem \ref{thm:6.1} 2) was announced in this form
 in \cite[Remark 2.10]{Kb2II}.  
See \cite{A, Kb4, SZ} for recent results without unitarity.  

\noindent
{\rm 2)}
If $H = K$ a maximal compact subgroup of $G$,
 then the assumption of Theorem~\ref{thm:6.1}~1)
 is obviously satisfied
 because $P_H = K$ and $K P = G$.
In this case $\dim \tau < \infty$
 for any $\tau \in \widehat K$.  
This simplest case gives an analytic proof 
 to the celebrated result of Harish-Chandra asserting that
{\sl{any irreducible unitary representation is admissible}}
 (using a theorem of I. Segal on the existence
 of infinitesimal characters
 of irreducible unitary representations).

\noindent
{\rm 3)}
Even if \eqref{eqn:PP} fails,
 it may happen
 that $\dim \Hom_H(\tau, \pi|_H)<\infty$
 for any $\tau \in \widehat H$
 for a specific triple $(\pi, G, H)$.  
This was studied in details in \cite{Kb2I,Kb2II,Kb2III}
 when the decomposition is discretely decomposable. 
\end{remark}

\section{Appendix}
\subsection{Proof of Lemma \ref{lem:algH}}

\begin{lemma}
\label{lem:7.1}
Let $H_c$ be a complex algebraic group 
 acting on a smooth complex variety $X$ by 
$\Psi:\ H_c\times X\ni(g,x)\mapsto gx\in X$. 
Then there exists a locally closed submanifold $Y$ of $X$ in the Zariski 
topology such that the following two conditions holds:

\noindent
{\rm 1)} $\Psi|_{H_c\times Y}$ is a submersion.

\noindent
{\rm 2)} $\# \{x \in Y: H_c x = H_c y\}$ is finite and does
not depend on $y\in Y$.
\end{lemma}
\begin{proof}
Let $\ell$ be the minimal codimension of the submanifold $H_cx$ for $x\in X$.
Fix $p\in X$ such that the codimension of $H_c p$ is equal to $\ell$.
Let $Y$ be an $\ell$-dimensional locally closed submanifold of $X$ through $p$ 
such that $d\Psi|_{H_c\times Y}$ is surjective at $(e,p)$.
By shrinking $Y$ if necessary,
 we may assume that 
 $d \Psi|_{H_c \times Y}$ surjects $T_yY$ 
 at $(e,y)$ for all $y \in Y$.  
Since the surjectivity of $d\Psi|_{H_c\times Y}$ at $(e,y)$ 
implies that of $d\Psi|_{H_c\times Y}$ at $(h,y)$ for any $h\in H$, 
$\Psi|_{H_c\times Y}$ is a submersion.  
Consider a locally closed subvariety
$$
 \tilde Y = (\operatorname{pr}\times\id_Y)\scirc
            (\Psi|_{H_c\times Y}\times\id_Y)^{-1}(\Delta Y)
$$
of $Y\times Y$,
where 
 $\operatorname{pr}$ 
is the second projection map of $H_c \times Y$ onto $Y$
and
$\Delta Y = \{(y,y)\in Y\times Y: y\in Y\}$.
By definition, 
$(x,y)\in\tilde Y$ if and only if $H_cx=H_cy$.
Since the fiber of the map $\pi:\ \tilde Y\ni(x,y)\mapsto x\in Y$ 
is discrete, there exist a positive number $m$ and a Zariski open subset $Y'$ 
of $Y$ such that $\pi|_{\pi^{-1}(Y')}$ is an $m$-fold 
covering map of $Y'$.
Then we have the lemma by replacing $Y$ by $Y'$.
\end{proof}
\begin{lemma}
\label{lem:7.2}
Suppose we are in the setting of Lemma \ref{lem:7.1}.  
If $H$ and $M$ are real forms of $H_c$ and $X$
 such that $H \cdot M \subset M$, 
 then there exists a locally closed 
submanifold $N$ of $M$ in the usual topology 
 satisfying the following two conditions:  

\noindent
{\rm 1)} $\Psi|_{H\times N}$ is a submersion of $H\times N$ to $M$, 

\noindent
{\rm 2)} $Hx\ne Hy$ for any $x$, $y\in N$ with $x\ne y$.
\end{lemma}
\begin{proof}
Put $N'=M\cap Y$. 
Owing to Lemma~\ref{lem:7.1}~2), 
the cardinality 
\[
     n=\sup_{x\in N'}\#\{y\in N': Hx = Hy\}
\]
 is finite and so
we can find $n$ different points $p_1,\ldots,p_n$ of $N'$
such that $Hp_1=\cdots=Hp_n$.
Let $U_i$ be open neighborhoods of $p_i$ in $N'$ which do not meet each other.
Then the open subset $N = \{p\in U_1: Hp\cap U_i\ne\emptyset\ \text{for }i=1,
\ldots,n\}$ of $N'$ is the required one.
\end{proof}

The following lemma 
 in the non-algebraic setting may also be useful
 for Theorem \ref{thm:3.1}.
\begin{lemma}
\label{lem:7.3}
Let $H$ be a Lie group acts on a manifold $M$.
Suppose there exists a locally closed submanifold $N$ of $M$ such that the map
$\Psi:H\times M\ni(h,x)\mapsto hx\in M$ satisfies\/ {\rm Lemma~\ref{lem:7.2}~1)} and

\noindent
{\rm 2)${}'$} $m_N(x)< \infty$ for any $x\in N$.

\noindent
Here we set $m_N(x) := \# \{y \in N: Hy= Hx\}$.  
Then the conditions\/ {\rm Lemma~\ref{lem:7.2}~1)} and\/ {\rm 2)} are satisfied by 
shrinking $N$ if necessary.
\end{lemma}
\begin{proof}
Put $U_i = \{x\in N: m_N(x) > i\}$ for $i=1,2,\ldots$.
Then $U_i$ are open subsets of $N$ because $\Psi(H\times U)\cap N$ is open
in $N$ for any open subset $U$ of $N$.
Put $V_i = N\backslash U_i$.
Since $\bigcup_i V_i = N$ by our assumption, Baire's category theorem says
that there exists $V_m$ having an inner point under the induced topology of 
$N$.
Replacing $N$ by the interior of $V_m$ and using the same argument
in the proof of Lemma~\ref{lem:7.2}, we have Lemma~\ref{lem:7.3}.
\end{proof}

\subsection{Proof of Proposition~\protect\ref{prop:4.2}}
\label{subsec:7.2}

We shall prove a uniform estimate of the multiplicity 
 of irreducible representations occurring in a principal series
 representation.

Suppose we are in the setting of Section \ref{sec:4}.
Let $\alpha_1,\ldots,\alpha_n$ be the fundamental system in 
$\Sigma(\mathfrak{j})^+$ 
 and $\omega_1,\cdots,\omega_n$ the corresponding fundamental weights.
By taking a covering group of $G$ if necessary,
 we may assume that
 $G$ is the real form of the simply connected complex Lie group $G_c$
 or its covering group,
 so that the fundamental representation $V_i$ with the highest weight 
 $\omega_i$ lifts to $G$.
For $\lambda\in\mathfrak{j}_c^*$ we put
$$
  \lambda = \sum_{i=1}^n \Lambda_i(\lambda)\omega_i
$$
and define
$$
 \operatorname{Re}\lambda = \sum_{i=1}^n
  (\operatorname{Re}\Lambda_i(\lambda))\omega_i.
$$

We will review the Jantzen--Zuckerman translation principle.
Let $\mathcal{F}_\lambda(\mathfrak{g},K)$ be the category of $({\mathfrak {g}},K)$-modules
 of finite length with a generalized infinitesimal character $\chi_\lambda$.
After conjugation by the Weyl group if necessary,
 we may assume $\lambda$ satisfies
\begin{equation}
\label{eqn:8.1}
  \operatorname{Re}\ang{\lambda}{\alpha}\le0\quad\text{for }
  \alpha\in\Sigma(\mathfrak{j})^+.
\end{equation}
For $V \in \mathcal{F}_\lambda(\mathfrak{g},K)$
 we define $\Phi_\lambda^i(V) :=p_{\lambda+\omega_i}(V \otimes V_i)$.
Here $p_{\lambda+\omega_i}$ is the projection map 
 to the primary component with generalized infinitesimal character 
 $\lambda+ \omega_i$.
Then $\Phi_\lambda^i$ is an exact functor
 from $\mathcal{F}_\lambda(\mathfrak{g},K)$ to $\mathcal{F}_{\lambda+\omega_i}(\mathfrak{g},K)$.
Similarly,
 we define a functor
 from $\mathcal{F}_{\lambda+\omega_i}(\mathfrak{g},K)$ 
 to $\mathcal{F}_{\lambda}(\mathfrak{g},K)$ by
 $\Psi_\lambda^i(W) := p_{\lambda}(W \otimes V_i^*)$,
 where $V_i^*$ is the contragredient representation of $V_i$.

Let $(\zeta, V_{\zeta}) \in \Lrep$.  
We write $d \zeta \in {\mathfrak {j}}_c^{\ast}$
 for the highest weight 
 with respect to $\Sigma ({\mathfrak {t}})^+$, 
 and take $w_o \in W(\mathfrak{j})$ such that
 $\lambda :=w_o d\zeta$ satisfies \eqref{eqn:8.1}.
Assume $\operatorname{Re}\Lambda_i(\lambda) < -1$ for some $i$.
This assumption assures that $\lambda$ and $\lambda + \omega_i$ 
 are equisingular,
 namely,
 $\ang{\lambda}{\alpha} = 0 \Leftrightarrow \ang{\lambda+\omega_i}{\alpha} =0$
 for $\alpha \in \Sigma({\mathfrak{j}})$.
Then we have an isomorphism of $(\mathfrak{g},K)$-modules:
\begin{equation}
\label{eqn:8.2}
  \Psi_\lambda^i(I_P^G(\zeta'))\simeq \ipgz
  \;
  \text{ and }
  \;
  I_P^G(\zeta')\simeq \Phi_\lambda^i (\ipgz).
\end{equation}
Here $(\zeta',V_{\zeta'}) \in\Lrep$ is the unique representation such that
 $V_{\zeta',P}$ occurs
 as a subquotient of $V_{\zeta,P}\otimes V_i$ and satisfies
 $w_o d\zeta'=\lambda+\omega_i$.
Thanks to \cite[Theorem 7.232]{KV},
  $\Phi_\lambda^i$ induces an equivalence of categories
 between $\mathcal{F}_\lambda(\mathfrak{g},K)$ and 
 $\mathcal{F}_{\lambda+\omega_i}(\mathfrak{g},K)$.
In particular, 
 $\Phi_\lambda^i$ sends (non-zero) irreducible $(\mathfrak{g},K)$-modules 
 to (non-zero) irreducible $(\mathfrak{g},K)$-modules
 and
 we have
\begin{alignat*}{1}
   \Hom_{(\mathfrak{g},K)}(\pi, \ipgz)
   &\simeq
   \Hom_{(\mathfrak{g},K)}(\Phi_\lambda^i(\pi), \Phi_\lambda^i (\ipgz))
\\
   &\simeq
   \Hom_{(\mathfrak{g},K)}(\Phi_\lambda^i(\pi), I_P^G(\zeta'))
\end{alignat*}
 for any $(\pi,V) \in \mathcal{F}_\lambda(\mathfrak{g},K)$.
Here we use \eqref{eqn:8.2} for the second equality.
Hence applying $\Phi_\lambda^i$ successively, we may assume 
$$
   |\operatorname{Re} d\zeta| \le C
$$
 in order to prove Proposition~\ref{prop:4.2},
 where $| \cdot |$ is the norm induced from the Killing form
 and $C := |\sum_{i=1}^n \omega_i|$.

Now we recall Vogan's results on minimal $K$-type theory.
We take a Cartan subalgebra $\tilde{\mathfrak{t}}$ of $\mathfrak{k}$
 and fix a positive system $\Delta^+(\mathfrak{k}_c, \tilde{\mathfrak{t}}_c)$.
We write $\delta_K \in \sqrt{-1}\tilde{\mathfrak{t}}^*$ for half the sum of
 elements in $\Delta^+(\mathfrak{k}_c, \tilde{\mathfrak{t}}_c)$.
If $\mu \in \tilde{\mathfrak{t}}_c^*$ is the highest weight of a $K$-type $\tau$
 we define $\Vert \tau \Vert := |\mu + 2 \delta_K|$,
 where $| \cdot |$ denotes the norm in $\sqrt{-1} \tilde{\mathfrak{t}}^*$
 induced from the Killing form.
A minimal $K$-type of the $(\mathfrak{g},K)$-module $(\pi,V)$ 
 is a $K$-type $\tau$ for which $|\tau|$ is minimal
  among all $K$-types occurring in $\pi$.
It follows from \cite[Theorem~10.26]{KV}  
 that
 there exists a constant $C'$ depending only on $\mathfrak{g}$ 
 with the following property:
 if $\pi$ is a $(\mathfrak{g},K)$-module with infinitesimal
 character $\lambda$, then
$$
  |\operatorname{Re} (\lambda)| \ge \Vert \tau \Vert -C'.
$$
Let $N$ be the maximal dimension of $\tau \in \widehat{K}$
 among all $K$-types $\tau$ with $\Vert \tau \Vert \le C + C'$.
We remark that $N$ depends only on the Lie algebra $\mathfrak{g}$.
For $\pi \in \widehat{G}_{\operatorname {ad}}$,
 let $\tau$ be one of its minimal $K$-types.
Because $\tau$ occurs in $\pi$ with multiplicity one,
 we have
$$
  [\pi: \ipgz]
  \le \dim \Hom_K (\tau, \ipgz).
$$
Then the right-hand side equals
 $\dim \Hom_M (\tau_{|M}, \zeta_{|M})$
 by the Frobenius reciprocity theorem.
Since
 $\dim \Hom_M (\tau_{|M}, \zeta_{|M}) \le N$, 
 we have proved Proposition~\ref{prop:4.2}.
\hfill\qed

\subsection*{Acknowledgments}

Parts of the results and the idea of the proof were delivered
 in various occasions including the workshop at Sandbjerg Gods' in Denmark in 1991 organized by N. V. Pedersen, 
Summer School on Number Theory in Nagano (Japan) in 1995 organized by F. Sato,
in Distinguished Sackler Lectues organized by J. Bernstein at Tel Aviv University (Israel) in May 2007, in the conference in honor of E. B. Vinberg's 70th birthday in Bielefeld (Germany) in July 2007 organized by H. Abels, V. Chernousov, G. Margulis, D. Poguntke, and K. Tent, in the 41th `Seminar Sophus Lie' in Erlangen (Germany) in July 2011 organized by K.-H. Neeb,
in `Analysis of Lie Group' at Max Planck Institute for Mathematics
 in Bonn (Germany) organized by B. Kr\"otz and H. Schlichtkrull in 2011, 
 and in `Lie Groups, Lie Algebras and their Representations'
 in November 2011 in Berkeley (USA) organized by J. Wolf.
 The authors are grateful to these organizers for warm hospitality
and to participants for their feedbacks.

The authors were partially supported
 by Mittag-Leffler institute (Djursholm), 
 Grant-in-Aid for Scientific Research
 (B) (22340026)  and (A)(20244008).

\tolerance=2000

\textsc{\small
	Toshiyuki Kobayashi,	Kavli IPMU (WPI) and Graduate School of Mathematical Sciences, the University of Tokyo, 
	Komaba, Tokyo 153-8914, Japan.
\\[\medskipamount]\indent
	Toshio Oshima,	Graduate School of Mathematical Sciences,
	the University of Tokyo,	Komaba, Tokyo 153-8914, Japan.  }

\end{document}